\documentclass{compositio}

\usepackage{amssymb,amsfonts,amsthm,amsmath,verbatim,lscape,color,tensor,url}
\usepackage{stmaryrd}
\usepackage[enableskew]{youngtab}
\usepackage[all]{xy}
\DeclareMathOperator{\Aut}{Aut}

\DeclareMathOperator{\im}{im}

\DeclareMathOperator{\Hom}{Hom}
\DeclareMathOperator{\colim}{colim}

\DeclareMathOperator{\Spec}{Spec}
\DeclareMathOperator{\Spf}{Spf}

\DeclareMathOperator{\Sub}{Sub}

\DeclareMathOperator{\Level}{Level}

\DeclareMathOperator{\Sum}{Sum}

\DeclareMathOperator{\fib}{fib}
\DeclareMathOperator{\cof}{cof}
\DeclareMathOperator{\Mod}{\mathrm{Mod}}
\DeclareMathOperator{\Cok}{Cok}

\theoremstyle{theorem}
\newtheorem{theorem}{Theorem}[section]
\newtheorem{proposition}[theorem]{Proposition}
\newtheorem{corollary}[theorem]{Corollary}
\newtheorem{lemma}[theorem]{Lemma}

\newtheorem*{theorem*}{Theorem}

\newtheorem*{corollary*}{Corollary}
\theoremstyle{definition}
\newtheorem{definition}[theorem]{Definition}

\newtheorem{example}[theorem]{Example}
\theoremstyle{remark}

\newtheorem{remark}[theorem]{Remark}

\makeatletter
\ifx\SK@label\undefined\let\SK@label\label\fi
 \let\your@thm\@thm
 \def\@thm#1#2#3{\gdef\currthmtype{#3}\your@thm{#1}{#2}{#3}}
 \def\mylabel#1{{\let\your@currentlabel\@currentlabel\def\@currentlabel
  {\currthmtype~\your@currentlabel}
 \SK@label{#1@}}\label{#1}}
 \def\myref#1{\ref{#1@}}
\makeatother

\newcommand{\powser}[1]{[\![#1]\!]}

\newcommand{\G}{\mathbb{G}}
\newcommand{\F}{\mathbb{F}}

\newcommand{\Q}{\mathbb{Q}}

\newcommand{\Z}{\mathbb{Z}}
\newcommand{\N}{\mathbb{N}}
\newcommand{\QZ}[1]{\Q_p/\Z_p^{#1}}

\newcommand{\al}{\alpha}

\newcommand{\lra}[1]{\overset{#1}{\longrightarrow}}
\newcommand{\lla}[1]{\overset{#1}{\longleftarrow}}
\newcommand{\Prod}[1]{\underset{#1}{\prod}}
\newcommand{\Oplus}[1]{\underset{#1}{\bigoplus}}
\newcommand{\Otimes}[1]{\underset{#1}{\bigotimes}}
\newcommand{\Coprod}[1]{\underset{#1}{\coprod}}
\newcommand{\Colim}[1]{\underset{#1}{\colim}}

\newcommand{\E}{E_{n}}

\newcommand{\hotimes}{\hat{\otimes}}
\newcommand{\Loops}{\mathcal{L}}

\newcommand{\upperRomannumeral}[1]{\uppercase\expandafter{\romannumeral#1}}
\DeclareMathOperator{\coker}{\mathrm{coker}}

\DeclareMathOperator{\Sp}{\mathrm{Sp}}
\DeclareMathOperator{\Top}{\mathrm{Top}}
\DeclareMathOperator{\sk}{\mathrm{sk}}
\newcommand{\bP}{\mathbb{P}}
\newcommand{\BP}{\mathbb{T}^{BP}}

\numberwithin{equation}{section}

\date{\today}

\begin{document}

\title{Brown--Peterson cohomology from Morava $E$-theory}
\author{Tobias Barthel}
\email{tbarthel@math.ku.dk}
\address{Department of Mathematical Sciences, University of Copenhagen, Universitetsparken 5, 2100 K{\o}penhavn {\O}, Denmark}
\author[Nathaniel Stapleton]{Nathaniel Stapleton, with an appendix by Jeremy Hahn}
\email{nat.j.stapleton@gmail.com}
\address{Fakult{\"a}t f{\"u}r Mathematik, Universit{\"a}t Regensburg, 93040 Regensburg, Germany}
%
%
\shortauthors{Tobias Barthel and Nathaniel Stapleton}


\begin{abstract}
We prove that the $p$-completed Brown--Peterson spectrum is a retract of a product of Morava $E$-theory spectra. As a consequence, we generalize results of Ravenel--Wilson--Yagita and Kashiwabara from spaces to spectra and deduce that the notion of good group is determined by Brown--Peterson cohomology. Furthermore, we show that rational factorizations of the Morava $E$-theory of certain finite groups hold integrally up to bounded torsion with height-independent exponent, thereby lifting these factorizations to the rationalized Brown--Peterson cohomology of such groups. 
\end{abstract}

\maketitle

\section{Introduction}

Many important cohomology theories $E$ are constructed from complex cobordism $MU$ or Brown--Peterson cohomology $BP$ via Landweber's exact functor theorem. Viewing this process as a simplification, one might wonder what kind of information about $BP^*(X)$ is retained in $E^*(X)$, for $X$ a space or a spectrum. Motivated by this question, the goal of this paper is two-fold: In the first part, we show that many properties of the $BP$-cohomology of a spectrum are determined by the collection of its Morava $K$-theories $K(n)^*(X)$. In the second part, transchromatic character theory is used to factor the rationalized $BP$-cohomology of classifying spaces of certain finite groups by establishing height-independent bounds on the torsion in Morava $E$-theory. This has the curious consequence that the $BP$-cohomology of finite groups behaves more algebro-geometrically than one might expect. 

Let $E_n$ denote Morava $E$-theory of height $n$, which is a Landweber exact $E_{\infty}$-ring spectrum with coefficients $E_n^*=W(\kappa)\llbracket u_1,\ldots,u_{n-1}\rrbracket[u^{\pm 1}]$. The key observation of this paper is a natural extension of a theorem due to Hovey~\cite[Thm.~3.1]{hoveycsc}, realizing $BP$ as a summand in a product of simpler cohomology theories. 

\begin{theorem*}
The $p$-completed Brown--Peterson spectrum $BP_p$ is a retract of $\prod_{n>0}E_n$, the product over all Morava $E$-theories $E_n$.
\end{theorem*}

The general idea is now as follows: The behavior of Morava $E$-theory at height $n$ is closely connected to the behavior of Morava $K$-theory $K(n)$ at the same height, and these cohomology theories determine each other in many cases. Starting from knowledge about the Morava $K$-theories $K(n)^*(X)$ of a spectrum $X$, we can then deduce properties of $E_n^*(X)$, which collectively control $BP_p^*(X)$ by the above splitting. Conversely, the $BP$-cohomology of \emph{spaces} determines their Morava $K$-theory. 

The class of spectra amenable to such comparison results are those with evenly concentrated Morava $K$-theory for all heights or, more generally, with Landweber flat $BP$-cohomology; examples abound. Using the previous theorem and certain base change formulas, we then obtain generalizations of the main structural results of \cite{rwy,rwy2} to spectra. 

\begin{theorem*}
Let $X$ be a spectrum. If $K(n)^*X$ is even for all $n > 0$, then $BP_p^*X$ is even and Landweber flat. Moreover, for $f\colon X \lra{} Y$ a map between spectra with even Morava $K$-theories, $BP_p^*(f)$ is injective or surjective if $K(n)^*(f)$ is injective or surjective for all $n$. 
\end{theorem*}
 
Moreover, we give conditions on a spectrum $X$ that implies that $BP_p^*(\bP X)$ is Landweber flat, where $\bP X$ denotes the free commutative algebra spectrum on $X$. Combining the retract theorem together with work of Rezk~\cite{rezkcongruence} as well as the first author with Frankland~\cite{frankland}, we construct a functor 
\[\BP\colon \Mod_{BP_p^*}^c \longrightarrow \Mod_{BP_p^*}.\]
For a large class of spaces, containing spheres and higher Eilenberg--Mac Lane spaces, we then show that $BP_p^*(\bP X)$ is functorially determined by $BP_p^*(X)$. 

\begin{theorem*}
There  exists a functor $\BP$ on the category $\Mod_{BP_p^*}^c$ of (topological) $BP_{p}^*$-modules such that, for $X$ a space with $K(n)_*(X)$ even and degreewise finite for all $n>0$, there is a natural isomorphism
\[\xymatrix{\BP BP_p^*(X) \ar[r]^-{\cong} & BP_p^*(\bP X)}\]
of $BP_p^*$-modules.
\end{theorem*}

Roughly speaking, this result says that the structure of the corresponding K{\"u}nneth and homotopy orbit spectral sequences for $BP_p^*(\bP X)$ is completely controlled by the topological module $BP_p^*(X)$. 

As another consequence of our results, we deduce that various notions of good groups coincide and are controlled by the $BP$-cohomology of the group. This motivates the second part of the paper, where we take the above ideas one step further. Transchromatic character theory establishes a link between Morava $E$-theories of different heights, leading to the question of which properties of $E_n^*(X)$ and $BP^*(X)$ are detected by height $1$ and height $0$, i.e., by topological $K$-theory and rational cohomology. It turns out that the connection is surprisingly strong, an observation that has already been exploited in~\cite{genstrickland, bscentralizers}.

The second part of the paper focuses on the rationalization of the retract theorem. We apply the resulting map to classifying spaces of finite groups. Algebro-geometrically, $\E^*(BA)$ for a finite abelian group $A$ corepresents the scheme $\Hom(A^*, \G_{\E})$ that parametrizes maps from the Pontryagin dual of $A$ to the formal group $\G_{\E}$ associated to Morava $E$-theory. Rationally, there is then a decomposition into schemes classifying level structures
\[\Hom(A^*, \G_{\E}) \cong_{\Q} \coprod_{H \subseteq A} \Level(H^*, \G_{\E}),  \]
where $H$ runs through the subgroups of $A$, see~\cite{subgroups}. Up to torsion, the scheme of level structures $\Level(H^*, \G_{\E})$ is corepresented by $\E^*(BH)/I$, where $I$ is the transfer ideal, i.e., the ideal generated by transfers from the maximal subgroups of $H$. We show that this statement holds integrally up to bounded integral torsion, where the exponent of the torsion is in fact independent of the height we work at. This uses a result about level structures proven by Jeremy Hahn, which forms the appendix to this paper. 

\begin{theorem*}
Let $A$ be a finite abelian group, then the exponent of the torsion in the cokernel of the natural map 
\[\E^*(BA) \lra{} \Prod{H \subseteq A} \E^*(BH)/I \] 
is bounded independent of the height $n$. 
\end{theorem*}

In order to prove this, we construct a variant of the transchromatic character maps of~\cite{tgcm,bscentralizers} from $E$-theory at height $n$ to height $1$, which allows for tight control over the torsion. Since products of torsion abelian groups with a common torsion exponent are torsion as well, it follows that the natural map
\[
\Prod{n} \E^*(BA) \lra{} \Prod{n} \Prod{H \subseteq A} \E^*(BH)/I
\]
is a rational isomorphism. It is possible to apply the retract theorem to immediately deduce a similar decomposition for the $BP$-cohomology of finite abelian groups. 

\begin{corollary*}
Let $A$ be a finite abelian group and let $I$ denote the transfer ideal, then the natural map
\[BP_p^*(BA) \lra{} \Prod{H \subseteq A} BP_p^*(BH)/I\]
is a rational isomorphism.
\end{corollary*}

Similarly, the scheme corepresented by $\E^*(B\Sigma_m)$ decomposes rationally into a product of subgroup schemes $\Gamma\Sub_{\lambda \vdash m}(\G_{\E})$, indexed by partitions $\lambda \vdash m$ of $m$. Again, we prove that this statement holds integrally up to globally bounded integral torsion.

\begin{theorem*}
The exponent of the torsion in the cokernel of the natural map 
\[
\E^*(B\Sigma_{m}) \lra{} \Prod{\lambda \vdash m}\Gamma \Sub_{\lambda \vdash m}(\G_{\E})
\] 
is bounded independent of the height $n$. Consequently, there is a rational isomorphism
\[
BP_{p}^*(B\Sigma_{m}) \lra{} \Prod{\lambda \vdash m} (BP_{p}^*(B\Sigma_{\lambda})/I_{\lambda})^{\Sigma_{\underline{a}}},
\]
where $I_{\lambda}$ is a certain transfer ideal.
\end{theorem*}

Finally, we give a further illustration of the methods of the paper by proving a well-known version of Artin induction for the $BP$-cohomology of good groups.

\subsection{Relation to the literature}

The question of when and how $E^*(X)$ can be computed algebraically from $BP^*(X)$ or, conversely, what kind of information about $BP^*(X)$ is retained in $E^*(X)$ has a long history. After pioneering work of Johnson and Wilson~\cite{jwbpdim, jwbpops} and Landweber~\cite{landweber}, these problems have been studied systematically in a series of papers by Ravenel, Wilson, and Yagita~\cite{rwy,rwy2} and Kashiwabara~\cite{kashiwabarabps2n,kashiwabarabpqx}. Their methods are based on a careful study of the associated Atiyah--Hirzebruch spectral sequences. As a consequence, the results contained in these papers are unstable, i.e., only valid for spaces rather than arbitrary spectra. 

The results of our paper generalize the main structural theorems of \cite{rwy,rwy2} from spaces to spectra, by replacing the Atiyah--Hirzebruch spectral sequence arguments by the above retract result. This is remarkable, as such extensions seemed previously impossible, see the remark following Theorem 1.8 in \cite[p.2]{rwy}. Furthermore, the retract result has the pleasant side effect of simplifying many of the proofs, at the cost of losing control over any explicit generators and relations descriptions. In particular, we are unable to recover the computations for spheres, Eilenberg--Mac Lane spaces, and other spaces appearing in certain $\Omega$-spectra given in the aforementioned papers. 

In contrast to this, passing from $BP^*(X)$ to $K(n)^*(X)$ is more subtle, due to the existence of counterexamples to the generalizations of certain base change formulas to arbitrary spectra. Here, our results do not improve on the results in \cite{rwy,rwy2}.

In~\cite{kashiwabarabpqx}, Kashiwabara studies the following question: To what extent does $BP_p^*(X)$ control $BP_p^*(QX)$ for $X$ a space. He shows that $BP_p^*(QX)$ inherits Landweber flatness from $BP_p^*(X)$ and constructs a functor $\mathcal{D}$ such that 
\[\xymatrix{\mathcal{D}BP_{p}^*(X) \ar[r]^-{\cong} & BP_{p}^*(\Omega^{\infty} X),}\]
is an isomorphism under certain algebraic conditions on $X$. Our description of $BP_{p}^*(\bP X)$ is related to Kashiwabara's result via the Snaith splitting; the precise relation between Kashiwabara's theorem and ours is, however, not completely clear at the moment. It should be noted, however, that his functor allows for explicit calculations, whereas $\BP$ involves the idempotent given by the retract and is thus much less computable. 

\subsection{Outline}

We begin in Section 2 by reminding the reader about the cohomology theories and basic concepts that are used in this paper. These theories are related to each other by various base change formulas, most of which are collected from the literature. The first goal of Section 3 is the proof of the result which the rest of our paper is based on, namely that $BP_p$ is a retract of a product of Morava $E$-theories. We then deduce our main theorems about the structure of the $BP$-cohomology of spectra with even Morava $K$-theory. Moreover, we discuss the $BP$-cohomology of free commutative algebras and show that various notions of good groups are equivalent.  

The second part of the paper starts in Section 4 with a toy example, obtaining the existence of a height independent torsion exponent in the case of finite cyclic groups and $\Sigma_p$ by an explicit calculation. A number of auxiliary commutative algebra results are proven before introducing a variant of the transchromatic character map. The rest of the section contains the proofs of the main theorems for Morava $E$-theory and their consequences for the rational $BP$-cohomology of abelian and symmetric groups. A key result for the case of abelian groups is deferred to an appendix, written by Jeremy Hahn.

\subsection{Notation and conventions}

Throughout this paper we fix a prime $p$ and work in the category $\Sp$ of $p$-local spectra. By space, we shall always mean a CW complex of finite type viewed as a suspension spectrum. The results in the first part of the paper are formulated and proven more generally for the theories $P(m)$, but the reader not familiar with these might want to specialize to $P(0)  =BP_p$. We will also always write $\QZ{n}$ for $(\QZ{})^n$.

\acknowledgements{
We would like to thank Tyler Lawson, Eric Peterson, Charles Rezk, Tomer Schlank, and Neil Strickland for helpful conversations, and the Max Planck Institute for Mathematics for its hospitality.}

\section{The Brown--Peterson spectrum and related cohomology theories}

After introducing the complex oriented ring spectra that will be used throughout the paper, we recall the notion of Landweber flatness relative to the theory $P(m)$ and prove that Landweber flat modules are closed under products. We then give base change formulas relating these cohomology theories. Although most of these results are well-known, we hope the reader might like to see them collected in one place.

\subsection{Preliminaries}

We recall some basic terminology and facts that will be used in this paper; for more details, see for example~\cite{rwy}. Let $BP$ be the Brown--Peterson spectrum with coefficients $BP^*=\Z_{(p)}[v_1,v_2,\ldots]$, where the $v_i$ are Hazewinkel generators of degree $|v_i|=-2(p^i-1)$ with $v_0=p$; as usual, none of the constructions or results in this paper depend on this choice of generators. By \cite{landweberideals}, the invariant prime ideals in $BP^*$ are given by $I_n = (p,v_1,\ldots,v_{n-1})$ for $n \ge 0$. 

Recall that $P(m)$ denotes the $BP$-module spectrum with coefficients 
\[P(m)^*= BP^*/I_m,\]
which can be constructed using Baas--Sullivan theory of singularities or the methods of \cite{ekmm}. We set $P(0) = BP_p$, the $p$-completion of $BP$. The ideals $I_{m,n}=(v_{m},\ldots,v_{n-1})$ for $n \ge m$ are then precisely the invariant prime ideals of $P(m)^*$.

The Landweber filtration theorem~\cite{landweber} for $BP_*BP$-comodules that are finitely presented as $BP_*$-modules admits the following generalization to $P(m)$, independently due to Yagita and Yoshimura.

\begin{theorem}[\cite{yagitaleft,yosimuraleft}]\mylabel{thm:generalizedlft}
Suppose $M$ is a $P(m)^*P(m)$-module which is finitely presented as a $P(m)^*$-module, then there exists a finite filtration of $M$ by $P(m)^*P(m)$-modules 
\[0 =M_0 \subseteq M_1 \subseteq \ldots  \subseteq M_s = M\]
with filtration quotients $M_{i+1}/M_i \cong P(m)^*/I_{m,n_i}$ up to a shift and for some $n_i\ge m$. 
\end{theorem}

\begin{remark}
The classical Landweber filtration theorem is usually stated for $BP$, not for $P(0) = BP_p$. However, this will not affect any of the arguments or statements appearing later in this paper. 
\end{remark}

From $P(m)$ we can construct, for $n\ge m$, spectra $E(m,n)$ by taking the quotient by the ideals $(v_{n+1},v_{n+2},\ldots)$ and then inverting $v_n$, i.e.,
\[E(m,n)^* = 
\begin{cases}
\Z_{p}[v_1,\ldots,v_n][v_n^{-1}] & \text{if}\, m=0, \\
\F_p[v_m,\ldots,v_n][v_n^{-1}] & \text{if}\, m>0.
\end{cases}
\]
In particular, if $m=n$, then $E(m,m) = K(m)$ is Morava $K$-theory, and if $m=0$, then $E(0,n) = E(n)$ is $p$-complete Johnson--Wilson theory. Moreover, we define $\hat{E}(m,n)$ to be the $K(n)$-localization of $E(m,n)$,
\[\hat{E}(m,n) = L_{K(n)}E(m,n),\]
where $L_{K(n)}$ denotes Bousfield localization with respect to $K(n)$. Finally, we denote by $E_n$ height $n$ Morava $E$-theory, which is a 2-periodic finite free extension of $\hat{E}(0,n)$ given by adjoining an element $u$ of degree 2 with $u^{1-p^n} = v_n$ and extending coefficients to the ring $W(\kappa)$ of Witt vectors over a perfect field $\kappa$ of characteristic $p$.

\subsection{Landweber flatness}

Recall the following definition from \cite{rwy,rwy2}.

\begin{definition}
A module $M \in \Mod_{P(m)^*}$ is called Landweber flat if it is a flat module for the category of $P(m)^*P(m)$-modules which are finitely presented as $P(m)^*$-modules, i.e., if the functor
\[\xymatrix{-\otimes_{P(m)^*} M \colon \Mod_{P(m)^*P(m)}^{\text{fp}} \ar[r] & \Mod_{P(m)^*}}\]
is exact.
\end{definition}

The Landweber filtration theorem for $P(m)$ easily implies the following characterization of Landweber flatness, see for example \cite[Thm.~3.9]{rwy}.

\begin{proposition}\mylabel{prop:lflatcriterion}
A $P(m)^*$-module $M$ is Landweber flat if and only if 
\[v_n\colon P(n)^* \otimes_{P(m)^*}M \longrightarrow P(n)^* \otimes_{P(m)^*}M\]
is injective for all $n\ge m$. 
\end{proposition}

In the next section, we need a closure property for the collection of Landweber flat modules. Since we do not know of a published reference for this fact, we include the proof. 

\begin{lemma}\mylabel{lem:lflatproducts}
The collection of Landweber flat $P(m)^*$-modules is closed under products.
\end{lemma}
\begin{proof}
Suppose $(M_i)_i$ is a collection of Landweber flat $BP^*$-modules and let $n\ge m$. By \myref{prop:lflatcriterion}, 
it suffices to show that the top map in the commutative diagram 
\[\xymatrix{(\prod_i M_i) \otimes_{P(m)^*} P(m)^*/I_{m,n} \ar[r]^{v_n} \ar[d]_{\epsilon_m} & (\prod_iM_i)\otimes_{P(m)^*} P(m)^*/I_{m,n} \ar[d]^{\epsilon_m}\\
\prod_i(M_i \otimes_{P(m)^*} P(m)^*/I_{m,n}) \ar[r]_{v_n} & \prod_i(M_i \otimes_{P(m)^*} P(m)^*/I_{m,n})}\]
is injective. Since $P(m)^*/I_{m,n}$ is a finitely presented $P(m)^*$-module, the vertical maps $\epsilon_m$ are isomorphisms. But, by assumption, the bottom map is injective, hence so is the top map.
\end{proof}

\begin{remark}
This result should be compared to a theorem of Chase: For $R$ a (not necessarily commutative) ring, the collection of flat $R$-modules is closed under arbitrary products if and only if $R$ is left coherent. For a discussion of products of flat modules and a proof of Chase's theorem, see \cite{lambook}.
\end{remark}

\subsection{Base change formulas}

This section collects a number of base change formulas for the cohomology theories we use in this paper. While most of the results are well-known, it is not always easy to locate references in the published literature. 

\begin{lemma}\mylabel{lem:frometok}
Suppose $n \ge m$ and let $X$ be a spectrum such that $\hat{E}(m,n)^*(X)$ is evenly concentrated and flat as an $\hat{E}(m,n)^*$-module, then
\[K(n)^*(X) \cong \hat{E}(m,n)^*(X) \otimes_{\hat{E}(m,n)^*} K(n)^*,\]
which is even as well. Conversely, if $K(n)^*(X)$ is even, then $\hat{E}(m,n)^*(X)$ is even and flat as an $\hat{E}(m,n)^*$-module.
\end{lemma}
\begin{proof}
We will give the proof for $m=0$ and Morava $E$-theory $E_n$; the argument and referenced results generalize easily to the case $m>0$. First note that $E_n^*(X)$ is automatically complete with respect to the maximal ideal of $E_n^*$, because the function spectrum $F(X,E_n)$ is $K(n)$-local. By \cite[Prop.~A.15]{frankland}, this means that $E_n^*(X)$ is flat as an $E_n^*$-module if and only if it is pro-free, i.e., if it is the completion of some free $E_n^*$-module. This reduces the claim to~\cite[Prop.~2.5]{hoveystrickland}.
\end{proof}

In order to relate $E(m,n)$-cohomology to $P(m)$-cohomology, we need to recall the construction of the completed tensor product, see for example~\cite{konoyagita}. 

\begin{definition}\mylabel{def:completetensor}
Let $k$ be a cohomology theory and consider two complete topological algebras $A,B$ over $k^*$, with filtrations given by a system of opens $\{A_r\}$ and $\{B_r\}$, respectively. The completed tensor product of $A$ with $B$ over $k^*$ is then defined as
\[A \hotimes_{k^*} B = \lim_r (A \otimes_{k^*} B)/J_r, \]
where $J_r$ is the ideal spanned by $A_r \otimes_{k^*} B + A \otimes_{k^*} B_r$.

If $X$ is space with skeletal filtration $\{\sk_rX\}$, then its cohomology $k^*(X)$ can be topologized using the system of fundamental neighbourhoods of $0$ given by $F_r(X) = \ker(k^*(X) \to k^*(\sk_rX))$. In the following, we will always consider this topology when working with completed tensor products. 
\end{definition}

\begin{remark}
In particular, note that the completed tensor product of even cohomology groups is also concentrated in even degrees. 
\end{remark}

The next result is known as Morava's little structure theorem; a proof in the generality we need is given in \cite[Prop.~1.9]{rwy2}, see also \cite{fsfg}.

\begin{proposition}\mylabel{prop:littlemorava}
If $X$ is a space, then 
\[E(m,n)^*(X) \cong P(m)^*(X) \hotimes_{P(m)^*} E(m,n)^* \]
for all $n\ge m >0$ or $m=0, n>0$, and similarly for the completed theories $\hat{E}(m,n)$. Moreover, if $P(m)^*(X)$ is Landweber flat, then $\hat{E}(m,n)^*(X)$ is flat over $\hat{E}(m,n)^*$ for all $n\ge m >0$.
\end{proposition}
\begin{proof}
The base change formula for both $E(m,n)$ and $\hat{E}(m,n)$ is contained in the aforementioned references, see for example~\cite[Cor.~8.24]{fsfg}. To show the second part of the claim, observe that, by \cite[Prop.~2.1]{rwy2}, the assumption implies that the ideal $I_{m,n}$ acts regularly on $\hat{E}(m,n)^*(X)$. Then the obvious extension of~\cite[Thm.~A.9]{hoveystrickland} from $E_n$ to arbitrary $\hat{E}(m,n)$ gives flatness. 
\end{proof}

\begin{example}\mylabel{rwyex}
This result does not hold for arbitrary spectra instead of spaces. As a counterexample, consider the connective Morava $K$-theory spectrum $k(n)$. While $K(n)^*k(n) \ne 0$, it can be shown \cite[Rem.~4.9]{rwy} that $P(n)^*k(n)=0$, hence the result fails for $m=n$. 
\end{example}

\begin{lemma}\mylabel{basechangepk}
For any spectrum $X$ with $P(m)^*(X)$ Landweber flat and $n\ge m$, there is a natural isomorphism 
\[P(n)^*(X) \cong P(n)^* \otimes_{P(m)^*} P(m)^*(X).\]
\end{lemma}
\begin{proof}
We prove this by induction on $n \ge m$, the statement for $n = m$ being trivial. Now let $n \ge m$ and recall that \myref{prop:lflatcriterion} shows that $P(m)^*(X)$ Landweber flat implies that 
\[v_n\colon P(n)^* \otimes_{P(m)^*} P(m)^*(X) \longrightarrow P(n)^* \otimes_{P(m)^*} P(m)^*(X)\]
is injective for all $n\ge m$. By construction of $P(n+1)$, there is an exact triangle
\begin{equation}\label{eq:pm}
\xymatrix{P(n)^*(X) \ar[rr]^{\rho_n} && P(n+1)^*(X) \ar@{-->}[ld] \\
& P(n)^*(X) \ar[lu]^{v_n},}
\end{equation}
where the dashed arrow is the connecting homomorphism of degree 1. Since $P(n)^*(X) \cong P(n)^* \otimes_{P(m)} P(m)^*(X)$ by the induction hypothesis, our assumption combined with \eqref{eq:pm} shows that $\rho_n$ is surjective and that 
\begin{align*}
P(n+1)^*(X) & \cong \coker(v_n\colon P(n)^*(X) \to P(n)^*(X)) \\
& \cong \coker(v_n\colon P(n)^* \otimes_{P(m)^*} P(m)^*(X) \to P(n)^* \otimes_{P(m)^*} P(m)^*(X)) \\
& \cong P(n+1)^* \otimes_{P(m)^*} P(m)^*(X),
\end{align*}
thereby proving the claim. 
\end{proof}
 
\begin{corollary}\mylabel{cor:knfrombp}
If $X$ is a space with Landweber flat $P(m)$-cohomology, then
\[
K(n)^*(X) \cong K(n)^* \hotimes_{P(n)^*}(P(n)^* \otimes_{P(m)^*} P(m)^*(X))
\]
for all $n\ge m > 0$ or $m=0, n>0$. In particular, under the same assumptions, if $P(m)^*(X)$ is even, then $K(n)^*(X)$ is even as well. 
\end{corollary} 
\begin{proof}
Let $X$ be a space such that $P(m)^*(X)$ is Landweber flat. From the previous lemma we get
\begin{equation}\label{eq:pmfrombp}
P(n)^*(X) \cong P(n)^* \otimes_{P(m)^*} P(m)^*(X)
\end{equation}
for all $n \ge m$. But \myref{prop:littlemorava} with $m=n$ shows that 
\[K(n)^*(X) \cong K(n)^* \hotimes_{P(n)^*} P(n)^*(X)\]
which together with \eqref{eq:pmfrombp} yields the claim.
\end{proof}

\begin{remark}
In~\cite[Eq.~(1.8)]{konoyagita}, Kono and Yagita show that $K(n)^*(X)$ is isomorphic to $K(n)^*\hotimes_{P(m)^*}P(m)^*(X)$ for a space $X$ with $BP^*(X)$ Landweber flat. We do not know if their methods generalize to imply that $K(n)^*(X) \cong K(n)^*\hotimes_{P(m)^*}P(m)^*(X)$ in the situation of \myref{cor:knfrombp}. 
\end{remark}

\section{The splitting and some consequences}

In this section, we first prove that $P(m)$ is a retract of an infinite product of $E$-theories, following closely the argument given in~\cite[Thm.~3.1]{hoveycsc}. Combined with the base change results of the previous section, this allows us to generalize the main structural results of~\cite{rwy,rwy2} from spaces to spectra. Since we do not require any analysis of the Atiyah--Hirzebruch spectral sequence for this, our arguments are rather short in comparison. In Section~\ref{sec:kash}, we construct an analogue of Rezk's algebraic approximation functor $\mathbb{T}^{E_n}$ for $BP_{p}$.  These functors are then used to prove a version of Kashiwabara's main result~\cite{kashiwabarabpqx} about the $BP_{p}$-cohomology of free commutative algebras. Finally, we give several equivalent characterizations of good groups, thereby showing that this notion is really a global one.

\subsection{$BP$ as a retract of a product of Morava $E$-theories}

Let $I\colon \Sp \to \Sp$ denote Brown--Comenetz duality, which is a lift of Pontryagin duality for abelian groups to the category of spectra. To be precise, if $X$ is a spectrum, then $I$ represents the functor 
\[X \mapsto \Hom(\pi_{0}X,\Q/\Z_{(p)})\]
and we set $IX = F(X,I)$. In \cite[Prop.~5.2]{margolisbook}, Margolis proves a non-existence result for f-phantom maps with target of the form $IX$; recall that a map $Y \to X$ of spectra is called f-phantom if $Z \lra{} Y \lra{} X$ is null for all finite spectra $Z$.

\begin{lemma}[Margolis]\mylabel{margolislemma}
Any $f$-phantom map with target $IX$ must be null. 
\end{lemma}

The next result is a natural generalization of \cite[Thm.~3.1]{hoveycsc}.

\begin{theorem}\mylabel{thm:generalretract}
Suppose that $D$ is a $p$-complete Landweber flat $P(m)$-module spectrum. If there is a morphism $f\colon P(m) \lra{} D$ such that the induced maps
\[\xymatrix{P(m)_*/I_{m,n} \ar[r] & D_*/I_{m,n}}\]
are injective for all $n>m$, then $f$ is a split inclusion of spectra.
\end{theorem}
\begin{proof}
We will prove the result for $m=0$ using the classical Landweber filtration theorem. The argument in the case $m>0$ uses~\myref{thm:generalizedlft}, but is easier as $P(m)$ is already $p$-complete. Let $F$ be the fiber of the map $BP_p \to D$ and 
\[F' = \fib(BP \lra{} D);\]
we need to check that $F \lra{} BP_p$ is null. Since $BP$ has degreewise finitely generated homotopy groups, $BP_p = I^2BP$. In order to show the claim, it thus suffices by \myref{margolislemma} to check that $F \lra{} BP_p$ is f-phantom. There is a commutative diagram 
\[\xymatrix{F' \ar[r] \ar[d] & BP \ar[r] \ar[d] & D \ar[d]^{\simeq} \\
F \ar[r] \ar[d] & BP_p \ar[r] \ar[d] & D \ar[d] \\
C_1 \ar[r] & C_2 \ar[r] & C_3}\]
with all rows and columns being cofiber sequences. Since $D$ is $p$-complete, we get $C_3=0$ and thus that the top row in the following diagram is a cofiber sequence
\[\xymatrix{F' \ar[r] \ar[d] & F \ar[r] \ar[d] & \cof(BP \lra{} BP_p) \simeq C_1 \ar@{-->}[ld]\\
BP \ar[r] & BP_p.}\]
If we can show that $F' \lra{} BP$ is f-phantom, then the composite $F' \lra{} BP \lra{} BP_p$ is also f-phantom, hence null by \myref{margolislemma}, so the indicated factorization in the right triangle exists. Because $\cof(BP \lra{} BP_p)$ is acyclic with respect to the mod $p$ Moore spectrum $M(p)$ and $BP_p$ is $M(p)$-local, the dashed map must be null, hence so is $F \lra{} BP_p$.

Therefore, Spanier--Whitehead duality reduces the claim to proving that, for $X$ a finite spectrum, the induced map 
\[\xymatrix{BP_*(X) \ar[r] & D_*(X)}\]
is injective. By the Landweber filtration theorem, the $BP_*BP$-comodule $BP_*(X)$ admits a finite filtration by comodules $F_i$ with filtration quotients $F_{i+1}/F_i \cong BP_*/I_{n_i}$ (up to suspensions). Since $BP_*/I_{n_i} \lra{} D_*/I_{n_i}$ is injective by assumption, Landweber exactness of $D$ and the snake lemma applied to the commutative diagram
\[\xymatrix{0 \ar[r] & F_i \ar[r] \ar[d] & F_{i+1} \ar[r] \ar[d] & BP_*/I_{n_i} \ar[r] \ar[d] & 0 \\
0 \ar[r] & F_i \otimes_{BP_*} D_* \ar[r] & F_{i+1} \otimes_{BP_*} D_* \ar[r] & BP_*/I_{n_i} \otimes_{BP_*} D_* \ar[r] & 0}\]
gives the claim inductively. 
\end{proof}

\begin{remark}
Note that the splitting constructed above is only additive, and not a map of ring spectra. 
\end{remark}

\begin{corollary}\mylabel{cor:bpretract}
The natural diagonal map 
\[\xymatrix{P(m) \ar[r] & \prod_{n \in I} \hat{E}(m,n)}\]
is a split inclusion for any infinite set $I$ of integers greater than $m$. In particular, $BP_p$ is a retract of $\prod_{n >0} E_n$. 
\end{corollary}
\begin{proof}
It suffices to check that $D_I=\prod_{n \in I} \hat{E}(m,n)$ satisfies the conditions of \myref{thm:generalretract}. The spectra $\hat{E}(m,n)$ are Landweber exact and $K(n)$-local, hence $p$-complete. By \myref{lem:lflatproducts}, $D_I$ is also Landweber flat, and because products of local objects are local, it is also $p$-complete. Since the standard maps $P(m) \lra{} \hat{E}(m,n)$ clearly make 
\[\xymatrix{P(m)_*/I_{m,k} \ar[r] & D_*/I_{m,k} = \prod_{n \in I} (\hat{E}(m,n)^*/I_{m,k})}\]
injective for all $k\ge m$, the claim follows. 
\end{proof}

\begin{remark}
The infinite product in \myref{cor:bpretract} cannot be replaced by a coproduct. In fact, every map  
\[\xymatrix{f\colon\bigvee_{n\in I} E_n \ar[r] &P(m)}\]
is null, for any $m\ge 0$. To see this, note that $P(m)$ is $H\F_p$-local, while $E_n$ is $H\F_p$-acyclic for all $n$ because the Morava $K$-theories are, hence $f$ must be null. 
\end{remark}

\begin{remark}
Combining \cite[Thm.~3.1]{hoveycsc} with \cite[Thm.~B]{hovsadofsky} yields a splitting similar to \myref{cor:bpretract}: The map
\begin{equation}\label{eq:hssplitting}
\xymatrix{BP_p \ar@{^{(}->}[r] & \prod_{n > 0}L_{K(n)}\left( \bigvee_{r \in S(n)} \Sigma^{r}L_{K(n)}E(n) \right)}
\end{equation}
is a split inclusion, where the $S(n)$ are explicit indexing sets of even integers. The spectrum on the right side is, however, much bigger than $\prod_{n>0} E_n$ and less convenient for taking cohomology due to the infinite wedge. In some sense, our splitting is as small as possible when using these techniques, but structural properties of $BP_p$ could also be deduced from the spectrum in \eqref{eq:hssplitting}.

Furthermore, since $MU$ splits $p$-locally into a wedge of even suspensions of $BP$, there are similar results for $MU_{(p)}$. 
\end{remark}

\subsection{Brown--Peterson cohomology and Morava $K$-theory}

In this section, we draw some consequences of \myref{cor:bpretract} on the kind of information the Morava $K$-theories detect about the $P(m)$-cohomology of spectra. These results generalize the main structural theorems of \cite{rwy} from spaces to spectra. 

\begin{remark}
While the authors remark in \cite[p.~2]{rwy} that their results are \emph{strictly unstable}, their counterexample refers to \myref{rwyex} and hence only affects \myref{prop:littlemorava}, which is not needed in our approach. 
\end{remark}

We start with a lemma that allows us to pass between ordinary and $p$-complete Brown--Peterson cohomology. 

\begin{lemma}\mylabel{lem:bpvsbpp}
If $X$ is a bounded below spectrum of finite type such that either
\begin{enumerate}
 \item $X$ is a suspension spectrum which is rationally equivalent to $S^0$, or
 \item $X$ is a spectrum which is rationally acyclic,
\end{enumerate}
then $BP^*(X)$ is even and Landweber flat if and only if so is $BP_p^*(X)$. 
\end{lemma}
\begin{proof}
Since $BP$ is connective, $BP_p$ is the Bousfield localization of $BP$ with respect to the mod $p$ Moore spectrum $M(p)$, and we get a fiber sequence
\begin{equation}\label{eq:mpbp}
\xymatrix{C_{M(p)}BP \ar[r] & BP \ar[r] & L_{M(p)}BP\simeq BP_p,}
\end{equation}
where $C_{M(p)}BP$ is rational and concentrated in odd degrees. On the one hand, if $X$ is a rationally acyclic spectrum, then $C_{M(p)}BP^*(X)=0$. On the other hand, if $X$ is a space rationally equivalent to the sphere, then $C_{M(p)}BP^*(X) \cong C_{M(p)}BP^*(S^0)$ is concentrated in odd degrees. Moreover, the connecting homomorphism $\delta$, fitting in the commutative diagram 
\[\xymatrix{BP_p^*(X) \ar[r]^-{\delta} \ar@{->>}[d] & C_{M(p)}BP^{*+1}(X) \ar[d]^{\cong} \\
BP_p^*(S^0) \ar@{->>}[r] & C_{M(p)}BP^{*+1}(S^0),}\]
is surjective. Therefore, in either case the long exact sequence associated to \eqref{eq:mpbp} yields isomorphisms
\[\xymatrix{BP^{\mathrm{odd}}(X) \ar[r]^-{\cong} & BP_p^{\mathrm{odd}}(X)}\]
and a short exact sequence
\[\xymatrix{0 \ar[r] & BP^{\mathrm{even}}(X) \ar[r] & BP_p^{\mathrm{even}}(X) \ar[r] & C_{M(p)}BP^{\mathrm{odd}}(X) \ar[r] & 0.}\]
This implies the claim about evenness. To see that $BP^*(X)$ is Landweber flat if and only if $BP^*_p(X)$ is Landweber flat, note that the finite type condition together with the Atiyah--Hirzebruch spectral sequence show that $BP^*(X)$ is finitely generated in each degree. By~\cite[Prop.~2.5]{bousfieldloc1}, we thus get an isomorphism
\[\xymatrix{BP^*(X) \otimes \Z_p \ar[r]^-{\cong} & BP_p^*(X),}\]
so the claim follows from \myref{prop:lflatcriterion} using flatness of $\Z_p$ and naturality. 
\end{proof}

In particular, this applies to finite spectra of type at least 1 and classifying spaces of finite groups. It will be implicitly used from now on, so whenever the assumptions are satisfied, $BP$ refers to the $p$-complete Brown--Peterson spectrum $BP_p$. The following theorem was proven for spaces in \cite[Thm.~1.8, Thm.~1.9]{rwy}. 

\begin{theorem}\mylabel{thm:bpfromkn}
Let $X$ be a spectrum. If $K(n)^*(X)$ is even for infinitely many $n$, then $P(m)^*(X)$ is even and Landweber flat for all $m$. In particular, \eqref{eq:pm} gives short exact sequences
\[\xymatrix{0 \ar[r] & P(m)^*(X) \ar[r]^{v_m} & P(m)^*(X) \ar[r] & P(m+1)^*(X) \ar[r] & 0.}\]
\end{theorem}
\begin{proof}
Let $I$ be the infinite set of natural numbers $n$ such that $K(n)^*(X)$ is even and $n\ge m$. By \myref{lem:frometok}, the assumption implies that $\hat{E}(m,n)^*(X)$ is even and flat over $\hat{E}(m,n)^*$. Since the $P(m)^*$-module $\hat{E}(m,n)^*$ is Landwever flat, so is $\hat{E}(m,n)^*(X)$ for all $n \in I$. It follows from \myref{lem:lflatproducts} that $\prod_{n \in I}\hat{E}(m,n)^*(X)$ is Landweber flat as well. Since Landweber flat modules are clearly closed under retracts, the claim now follows from \myref{cor:bpretract}. 
\end{proof}

\begin{remark}
More generally, the same proof shows that any spectrum $X$ with $K(n)^*(X)$ concentrated in degrees divisible by a fixed integer $d$ for infinitely many $n$ has $P(m)^*X$ also concentrated in degrees divisible by $d$ for all $m$. Similar observations apply to the next corollary; this should be compared to \cite{minamikn}.
\end{remark}

As an immediate consequence, we recover~\cite[Thm.~1.2]{rwy}, which complements \myref{thm:bpfromkn} for spaces. By \myref{rwyex}, this result does not generalize to arbitrary spectra. 

\begin{corollary}\mylabel{cor:kneven}
If $X$ is a space with $K(n)^*(X)$ even for infinitely many $n$, then $K(m)^*(X)$ is even for all $m > 0$. 
\end{corollary}
\begin{proof}
By \myref{thm:bpfromkn} for $m=0$, $P(0)^*(X)$ is even, hence so is $K(m)^*(X)$ for all $m > 0$ by \myref{cor:knfrombp}.
\end{proof}

\begin{example}
The spectrum $\Sigma H\Q$ has trivial Morava $K$-theory for all heights $n>0$, but $K(0)^1\Sigma H\Q \ne 0$. For a more interesting example, let $X = K(\Z,3)$, then the rational cohomology $K(0)^*(X) \cong \Q[x]/x^2$ with $x$ in degree 3, but $K(n)^*(X)$ is even and non-trivial for all $n\ge 2$ by \cite{RW}. Therefore, the conclusion of the corollary cannot be extended to $m=0$. 
\end{example}

\begin{corollary}\mylabel{cor:pmeven}
If $X$ is a space such that $P(m)^*(X)$ is even for some $m\ge 0$, then $K(n)^*(X)$ is even for all $n > 0$. 
\end{corollary}
\begin{proof}
If $P(m)^*(X)$ is even, it is also Landweber flat as the connecting homomorphism in \eqref{eq:pm} must be zero for degree reasons. Therefore, \myref{cor:knfrombp} shows that $K(n)^*(X)$ is even for all $n \ge m+1$, thus \myref{cor:kneven} applies.
\end{proof}

We will need the following lemma about maps of complete modules; a proof in the case $m=n$ can be found in Hovey's unpublished notes~\cite{hoveynotes} or \cite{frankland}, which requires only minor modifications for the general case.  

\begin{lemma}\mylabel{lem:injsurj}
Suppose $M,N$ are flat $\hat{E}(m,n)^*$-modules which are complete  with respect to the maximal ideal $I_{m,n}$ of $\hat{E}(m,n)^*$. A map $f\colon M \lra{} N$ is injective or surjective if and only if $f \otimes_{\hat{E}(m,n)^*} K(n)^*$ is injective or surjective, respectively. 
\end{lemma}

The next theorem generalizes \cite[Thm.~1.17, Thm.~1.18]{rwy} to all spectra. 

\begin{theorem}\mylabel{thm:pmexact}
Let $f \colon X \lra{} Y$ be a map of spectra such that both $K(n)^*(X)$ and $K(n)^*(Y)$ are even for $n \in I$ with $I \subset \N$ infinite. 
\begin{enumerate}
 \item If $f^*\colon K(n)^*(Y) \lra{} K(n)^*(X)$ is injective (surjective) for all $n \in I$, then so is \[f^*\colon P(m)^*(Y) \lra{} P(m)^*(X)\]
 for all $m$.
 \item Suppose $g\colon Y \lra{} Z$ is another map with $K(n)^*(Z)$ even for all $n \in I$ and such that $g \circ f \simeq 0$. If 
 \[\xymatrix{K(n)^*(Z) \ar[r]^{g^*} & K(n)^*(Y) \ar[r]^{f^*} & K(n)^*(X) \ar[r] & 0}\]
 is an exact sequence for all $n \in I$, then so is 
 \[\xymatrix{P(m)^*(Z) \ar[r]^{g^*} & P(m)^*(Y) \ar[r]^{f^*} & P(m)^*(X) \ar[r] & 0,}\]
 for all $m \ge 0$.
\end{enumerate}
\end{theorem}
\begin{proof}
Fix some integer $m \ge 0$. We will prove the surjectivity statement; the argument for injectivity is analogous. By assumption and \myref{lem:frometok}, there is a commutative diagram
\[\xymatrixcolsep{6pc}\xymatrix{K(n)^*(Y) \ar[r]^-{K(n)^*f} \ar[d]_{\cong} & K(n)^*(X) \ar[d]^{\cong} \\
\hat{E}(m,n)^*(Y) \otimes_{\hat{E}(m,n)^*} K(n)^* \ar[r]_-{f^* \otimes_{\hat{E}(m,n)^*} K(n)^*} & \hat{E}(m,n)^*(X) \otimes_{\hat{E}(m,n)^*} K(n)^*.}\]
Since $K(n)^*f$ is surjective, \myref{lem:injsurj} implies that $\hat{E}(m,n)^*f$ is surjective for $n \in I$ as well. Products in the category of modules are exact, so the bottom map in the commutative square is also surjective:
\[\xymatrix{P(m)^*(Y) \ar[r]^-{f^*} \ar[d] & P(m)^*(X) \ar[d] \\
\prod_{n \in I} \hat{E}(m,n)^*(Y) \ar@{->>}[r]_-{f^*} & \prod_{n \in I} \hat{E}(m,n)^*(X).}\]
Retracts of surjective maps are surjective, hence \myref{cor:bpretract} yields the claim.

Part (2) follows formally from (1) as in the proof of \cite[Thm.~1.18]{rwy}.
\end{proof}

\subsection{Applications to the cohomology of free commutative algebras}\label{sec:kash}

In \cite{kashiwabarabps2n,kashiwabarabpqx}, Kashiwabara studies the question of when and how $BP_{p}^*(\Omega^{\infty}X)$ is determined by $BP_{p}^*(X)$, for $X$ a space or spectrum. In particular, he considers two variants $\mathcal{K}'_{{}_0BP}$ and $\mathcal{M}'_{BP}$ of the category of augmented unstable $BP$-cohomology algebras and the category of stable $BP_{p}$-cohomology algebras, respectively, and shows that there is a left adjoint 
\[\xymatrix{\mathcal{D}\colon \mathcal{M}'_{BP} \ar@<0.5ex>[r] & \mathcal{K}'_{{}_0BP}\colon \mathcal{I} \ar@<0.5ex>[l]}\]
to the augmention ideal functor $\mathcal{I}$. This adjunction should be thought of as being analogous to the adjunction $(\Sigma^{\infty} \dashv \Omega^{\infty})$ between spaces and spectra. He then shows in~\cite[Thm.~0.11]{kashiwabarabpqx} that, for $X$ a space, $BP_{p}^*(QX)$ inherits Landweber flatness from $BP_{p}^*(X)$. Furthermore, if $X$ is a connected space satisfying certain conditions, there is a natural isomorphism
\[\xymatrix{\mathcal{D}BP_{p}^*(X) \ar[r]^-{\cong} & BP_{p}^*(QX),}\]
see~\cite[Thm.~0.12]{kashiwabarabpqx}. 

In this section, we will prove some extensions and variants of Kashiwabara's results. To this end, recall that, if $X$ is a connected space, the Snaith splitting provides an equivalence
\[\Sigma^{\infty}QX \simeq \bP\Sigma_+^{\infty}X,\]
where $\bP Y$ denotes the free commutative algebra on a spectrum $Y$. Thus our results, which are formulated in terms of $\bP$, are readily translated and compared to Kashiwabara's. 

Fix a height $n$ and recall the completed algebraic approximation functor of~\cite{rezkcongruence, frankland}: This is an endofunctor ${\mathbb{T}}^{E_n}$ on the category of complete $E_n^*$-modules together with a natural comparison map
\[\alpha_n(M)\colon {\mathbb{T}}^{E_n}\pi_*L_{K(n)}M \lra{} \pi_*L_{K(n)}\bP^{E_n}M,\]
where $M$ is an $E_n$-module and $\bP^{E_n}$ denotes the free commutative $E_n$-algebra functor. By \cite[Prop.~4.9]{rezkcongruence}, $\alpha_n(M)$ is an isomorphism whenever $M^*$ is flat over $E_n^*$. By~\cite[Prop.~3.9, Prop.~A.15]{frankland}, ${\mathbb{T}}^{E_n}$ preserves the category of flat $E_n^*$-modules, and it preserves evenness, see~\cite[3.2(7)]{rezkkoszul}. Moreover, there is natural decomposition of functors,
\[\mathbb{T}^{E_n} \cong \bigvee_{d \ge 0} \mathbb{T}^{E_n}_d, \]
corresponding to the analogous decomposition of $\bP^{E_n}$.

We are now in the position to prove a generalization of~\cite[Thm.~0.11]{kashiwabarabpqx} to spectra. Note, however, that Kashiwabara assumes that $BP_p^*(X)$ is Landweber flat (without evenness) to deduce that $BP_p^*(\bP X)$ is also Landweber flat, so our result is only a partial generalization. 

\begin{proposition}
If $X$ is a spectrum with even Morava $K$-theory for infinitely many $n$, then $BP_p^*(\bP X)$ is even and Landweber flat.
\end{proposition}
\begin{proof}
Let $n > 0$ be such that $K(n)_*(X)$ is even. Recall that the completed $E$-homology of $X$ is defined as $E_{n,*}^{\vee}(X) = \pi_*L_{K(n)}(E_n\wedge X)$. The assumption on $X$ implies that $E_{n,*}^{\vee}(X)$ is flat and even by the homological version of \myref{lem:frometok}, see \cite[Prop.~8.4(f)]{hoveystrickland}, thus the previous discussion gives an isomorphism
\[\xymatrix{\alpha_n(E_n \wedge X)\colon {\mathbb{T}}^{E_n}E_{n,*}^{\vee}(X) \ar[r]^-{\cong} & E_{n,*}^{\vee}(\bP X)}\]
of flat and evenly concentrated $E_n^*$-modules, using the equivalence $E_n \wedge \bP X \simeq \bP^{E_n}(E_n \wedge X)$. It follows from \myref{lem:frometok} that $K(n)^*(\bP X)$ is even as well, so \myref{thm:bpfromkn} finishes the proof.
\end{proof}

Denote by 
\[
\epsilon\colon \Prod{n>0}E_n \to \Prod{n>0}E_n
\] 
the idempotent given by \myref{cor:bpretract}; note that $\epsilon$ is not necessarily unique, but we will fix one throughout this section. 

\begin{definition}
The $BP$-based algebraic approximation functor $\BP_d$ of degree $d\ge 0$ is constructed as the functor
\[\BP_d :=  \epsilon\prod_{n>0}\mathbb{T}^{E_n}_d (E_n^* \hotimes_{BP_p^*} -) \colon \Mod_{BP_p^*}^c \longrightarrow \Mod_{BP_p^*}\]
on the category of $BP_p^*$-modules equipped with a complete topology as in \myref{def:completetensor}. We then define the total algebraic approximation functor as $\BP = \prod_{d\ge 0} \BP_d$.
\end{definition}

This allows us to prove an analogue of Kashiwabara's result~~\cite[Thm.~0.12]{kashiwabarabpqx} exhibiting a class of spaces $X$ for which the $BP_p^*$-cohomology of $\bP X$ is functorially determined by the topological module $BP_p^*(X)$. 

\begin{theorem}\mylabel{thm:bpp}
Let $X$ be a space with $K(n)_*(X)$ even and degreewise finite for infinitely many $n$, then there exists a natural isomorphism
\[\xymatrix{\BP BP_p^*(X) \ar[r]^-{\cong} & BP_p^*(\bP X)}\]
of $BP_p^*$-modules.
\end{theorem}
\begin{proof}
Let $S$ be the set of those natural numbers $n$ such that  $K(n)_*(X)$ is even and degreewise finite. The free commutative algebra functor decomposes into its degree $d$ constituents,
\[\bP(-) \simeq \bigvee_{d\ge0}\bP_d(-) \simeq \bigvee_{d\ge0} (-)^{\wedge d}_{h\Sigma_d},\] 
so we obtain a natural commutative diagram 
\[\xymatrix{BP_p^*(\bP X) \ar[r]^-{\cong} \ar[d] & \prod_dBP_p^*(\bP_dX) \ar[d]\\
\prod_{n \in S}E_n^*(\bP X) \ar[r]_-{\cong} & \prod_d\prod_{n \in S}E_n^*(\bP_d X)}\]
for any spectrum $X$. Since $\BP$ is compatible with this decomposition, we can reduce to a fixed degree $d \ge 0$. 

There are natural isomorphisms
\begin{align*}
\BP_d BP_p^*(X) &= \epsilon\prod_{n\in S}\mathbb{T}^{E_n}_d (E_n^* \hotimes_{BP_p^*} BP_p^*(X)) \\
& \cong \epsilon\prod_{n\in S}\mathbb{T}^{E_n}_d (E_n^*(X)) & & \text{by \myref{prop:littlemorava}} \\
& \cong \epsilon\prod_{n\in S}\mathbb{T}^{E_n}_d \pi_*E_n^X  \\
& \cong \epsilon\prod_{n\in S}\pi_*L_{K(n)}\bP^{E_n}_dE_n^X & & \text{via } \prod_{n \in S} \alpha_n(E_n^X), 
\end{align*}
so it is enough to understand $\pi_*L_{K(n)}\bP^{E_n}_dE_n^X$ in terms of $E_n^*(\bP_d X)$. By~\cite[Thm.~8.6]{hoveystrickland}, a spectrum $X$ has degreewise finite $K(n)_*(X)$ if and only if $X$ is dualizable in the $K(n)$-local category; we write $D_{K(n)}$ for $K(n)$-local duality. Therefore, we obtain:
\begin{align*}
\pi_*L_{K(n)}\bP^{E_n}_dE_n^X & \cong \pi_*L_{K(n)}\bP^{E_n}_d(E_n \wedge D_{K(n)}X)  \\
& \cong \pi_*L_{K(n)}(E_n \wedge \bP_d D_{K(n)}X) \\ 
& \cong \pi_*L_{K(n)}(E_n \wedge D_{K(n)}\bP_d X) \\ 
& \cong \pi_*L_{K(n)}E_n^{\bP_d X}  \\ 
& \cong E_n^*(\bP_d X). 
\end{align*}
Here, the third isomorphism uses the fact~\cite{greenlees-sadofsky} that homotopy orbits agree with homotopy fixed points with respect to a finite group $G$ in the $K(n)$-local category, i.e., $Y_{hG} \lra{\simeq} Y^{hG}$ $K(n)$-locally. Moreover, the fourth isomorphism can be seen as follows: Because $X$ is dualizable, $K(n)_*(X)$ is degreewise finite, hence so is $K(n)_*(\bP_d X)$ as $K(n)_*(B\Sigma_d)$ is degreewise finite. 
Using~\cite[Thm.~8.6]{hoveystrickland} once more, we see that $\bP_d X$ is also $K(n)$-locally dualizable, giving the fourth isomorphism above. 

Putting the pieces together, we obtain
\begin{align*}
\BP_d BP_p^*(X) & \cong \epsilon\prod_{n\in S}\pi_*L_{K(n)}\bP^{E_n}_dE_n^X \\
& \cong \epsilon\prod_{n\in S}E_n^*(\bP_d X) \\
& \cong BP_p^*(\bP_d X),
\end{align*}
hence the claim. 
\end{proof}

Note that having degreewise finite Morava $K$-theory is less restrictive than one might think. For example, all finite CW complexes satisfy this condition. For a different class of examples including classifying spaces of finite groups and Eilenberg--Mac Lane spaces $K(\Z/p,l)$, recall that a $\pi$-finite space is a space with only finitely many nonzero homotopy groups all of which are finite. By~\cite[Cor.~8.8]{hoveystrickland}, $\pi$-finite spaces are $K(n)$-locally dualizable, hence have degreewise finite Morava $K$-theory as well. Finite CW complexes with even cells and many $\pi$-finite spaces satisfy the evenness condition as well.  

We should, however, remark that we do not know the precise relation between our condition and Kashiwabara's assumption of well-generated $BP_{p}$-cohomology. In fact, Kashiwabara's result is somewhat sharper, in that his functor $\mathcal{D}$ is computable in an entirely algebraic fashion, as demonstrated in \cite{kashiwabarabpqx}. 

\begin{remark}
The hypotheses in \myref{thm:bpp} can be weakened if one is willing to work with a version of $\BP$ that incorporates a continuous dual. To be more precise, Hovey~\cite[Thm.~5.1]{hoveyoperations} shows that, for spectra $X$ with flat completed $E_n$-homology $E_{*,n}^{\vee}(X)$, there is a natural isomorphism
\[\xymatrix{E_n^*(X) \ar[r]^-{\cong} & \Hom_{E_n^*}(E_{n,*}^{\vee}(X),E_n^*).}\]
Passing from cohomology to homology is more subtle, however, and requires that we take into account the topology on $E_n^*(X)$ coming from the skeletal filtration on a CW spectrum $X$. In~\cite[Section 7]{stricklandfp}, Strickland proves
\[\xymatrix{E_{n,*}^{\vee}(X) \ar[r]^-{\cong} & \Hom_{E_n^*}^c(E_n^*(X),E_n^*)}\] 
when $X$ is a CW spectrum with free $E_*(X)$, where $\Hom^c$ denotes the module of continuous homomorphisms. Together, these two isomorphisms can be used to translate the statement into one about completed homology, to which the algebraic approximation functor $\mathbb{T}^{E_n}$ applies directly. 
\end{remark}

We end this section by giving a different criterion for when $P(m)^*(X)$ is even and Landweber flat. Let $\Phi_n\colon \Top \lra{} \Sp$ be the Bousfield--Kuhn functor, which was constructed in \cite{bousfieldbkf} for $n=1$ and, using the periodicity theorem~\cite{nilp1,nilp2}, for arbitrary $n$ in \cite{kuhnbkf}. This functor gives a factorization
\[L_{K(n)}X \simeq \Phi_n\Omega^{\infty}X\]
for any spectrum $X$; in fact, this can be improved to factor the telescopic localization in the analogous way, but we will not need this here. For $X \in \Sp$, Kuhn~\cite{kuhnaqgt} uses $\Phi_n$ to produce a (weak) map 
\[s_n(X)\colon L_{K(n)}\bP X \lra{} L_{K(n)}\Sigma^{\infty}_+\Omega^{\infty}X\]
which induces a natural monomorphism
\[s_n(X)_*\colon K(n)_*(\bP X) \lra{} K(n)_*(\Omega^{\infty}X).\]
As an easy consequence, we get:

\begin{proposition}
Let $X$ be a spectrum such that $K(n)_*\Omega^{\infty}X$ is even for infinitely many $n$, then $P(m)^*(\bP X)$ is even and Landweber flat for all $m$. 
\end{proposition}
\begin{proof}
By assumption, Kuhn's map $s_n(X)_*$ shows that $K(n)_*(\bP X)$ is even, hence \myref{thm:bpfromkn} applies. 
\end{proof}

\subsection{Equivalent characterizations of good groups}

Now let $G$ be a finite group. For the purposes of this paper, a group $G$ is called 
good if the Morava $K$-theories $K(n)^*(BG)$ are even for all $n\ge 0$. Note that this notion of good is a priori weaker than the one given in \cite{hkr}, but conjectured to be equivalent. The following definition first appeared in \cite{schusteryagitabp} for $m=0$. 

\begin{definition}
A finite group $G$ is called $P(m)$-good if the $P(m)$-cohomology of its classifying space $BG$ is concentrated in even degrees and is Landweber flat. 
\end{definition}

Note that it follows from \myref{lem:bpvsbpp} that a group is $P(0)$-good if and only if $BP^*(BG)$ is even and Landweber flat. We now see that the various notions of good groups coincide, thereby revealing the global nature of goodness for finite groups. 

\begin{proposition}
For a finite group $G$, the following conditions are equivalent:
\begin{enumerate}
 \item $G$ is good,
 \item $K(n)^*(BG)$ is even for infinitely many $n\ge 0$,
 \item $G$ is $BP_p$-good,
 \item $G$ is $P(m)$-good for some $m\ge 0$.
\end{enumerate}
\end{proposition}
\begin{proof}
The implications $(1) \implies (2)$ and $(3) \implies (4)$ are trivial. \myref{thm:bpfromkn} gives $(2) \implies (3)$, while $(4) \implies (1)$ follows from \myref{cor:pmeven}.
\end{proof}

This result generalizes work of Kono and Yagita in~\cite{konoyagita} for finite groups.

\section{Bounded torsion results for the Morava $E$-cohomology of finite groups}

In this section we study the rationalization of the $BP$-cohomology of certain finite groups by working with the rationalization of the split inclusion of \myref{cor:bpretract}
\[
\Q \otimes BP_{p}^{*}(X) \lra{} \Q \otimes \Prod{n} \E^*(X).
\]
There is a close analogy to number theory. The right hand side is a kind of $p$-local adeles in the stable homotopy category. Just as in number theory, we use these $p$-local adeles to study more global phenomena (the left hand side). The purpose of this section is to produce a factorization of $\Q \otimes BP_{p}^*(X)$ when $X = BA$ or $X = B\Sigma_m$ by lifting a factorization of $\Q \otimes \prod_{n} \E^*(X)$. For instance, in the case $X = BA$, we factor $\Q \otimes \Prod{n} \E^*(BA)$ by proving that the cokernel of the canonical map 
\[
\Prod{n} \E^*(BA) \lra{} \Prod{n} \Prod{H \subseteq A} \E^*(BH)/I
\]
has bounded torsion, where $I$ is the image of a transfer map. In this section, we will always work with the $p$-complete version of the Brown--Peterson spectrum; for simplicity we shall write $BP = BP_p$ for this spectrum.

\subsection{Cyclic groups and $\Sigma_p$ --- a toy case}
In this subsection we present some elementary observations that provide evidence for general bounded torsion results. Therefore, this subsection works backwards from the subsections that follow it. It uses $BP$-cohomology to give bounded torsion results for Morava $E$-theory.

Recall that, after choosing a coordinate, there is an isomorphism
\begin{equation} \label{ptorsion}
BP^*(B\Z/p^k) \cong BP^*\powser{x}/[p^k](x),
\end{equation}
where $[p^k](x)$ is the $p^k$-series for the formal group law associated to $BP$.

For any formal group law, it is standard to set
\[
\langle p^k \rangle (x) = [p^k](x)/[p^{k-1}](x) = \langle p \rangle ([p^{k-1}](x))
\]
so that
\[
BP^*(B\Z/p^k)/I \cong BP^*\powser{x}/\langle p^k \rangle(x),
\]
where $I \subset BP^*(B\Z/p^k)$ is the image of the transfer along $\Z/p^{k-1} \subset \Z/p^k$. This fact is a consequence of Quillen's \cite[Proposition 4.2]{Quillenelementary}. There is a canonical map
\[
BP^*(B\Z/p^k) \lra{} \Prod{0 \leq j \leq k} BP^*(B\Z/p^j)/I
\]
given by sending $1$ to $(1,\ldots,1)$ and $x$ to $(0,x,x,\ldots,x)$.

The following proposition is well-known.
\begin{proposition}
There is an isomorphism
\[
\Q \otimes BP^*(B\Z/p^k) \lra{\cong} \Q \otimes \Prod{0 \leq j \leq k} BP^*(B\Z/p^j)/I.
\]
\end{proposition}
\begin{proof}
For any formal group law there is a factorization
\[
[p^k](x) = \Prod{0 \leq j \leq k} \langle p^j \rangle (x),
\]
and we will show that the factors are coprime in the rationalization; the Chinese remainder theorem will then imply the claim. We may write $\langle p \rangle (x) = p + f(x)$ for some power series $f(x)$ with $x | f(x)$. This implies that
\[
[p^j](x) = [p^{j-1}](x)(p + f([p^{j-1}](x))),
\]
hence
\begin{equation}\label{eq:pjseries}
\langle p^j \rangle (x) = p + f([p^{j-1}](x)).
\end{equation}
But now if $t < j$, we have $\langle p^t \rangle (x) | [p^{j-1}](x)$ by definition, so that
\[
\langle p^t \rangle (x) | f([p^{j-1}](x))
\]
and using \eqref{eq:pjseries}
\[
p =  \langle p^j \rangle (x) - \langle p^t \rangle (x) \frac{f([p^{j-1}](x))}{\langle p^t \rangle (x)}.
\]
\end{proof}

The Morava $E$-cohomology of $B\Z/p^k$ satisfies an isomorphism as in Equation \ref{ptorsion}. Let $\G_{\E}$ be the formal group associated to $\E$. Given a coordinate $x$ on $\G_{\E}$, there is an isomorphism
\[
\E^*(B\Z/p^k) \cong \E^*\powser{x}/[p^k](x),
\]
where $[p^k](x)$ is the $p^k$-series for the formal group law $x+_{\G_{\E}}y$ induced by the coordinate. It is worth noting that the Weierstrass preparation theorem implies that this ring is a free $\E^*$-module of rank $p^{kn}$. This is very different from $BP^*(B\Z/p^k)$, which is infinitely generated as a $BP^*$-module.

\begin{corollary} \mylabel{cyclicdecomp}
There is an isomorphism
\[
\Q \otimes \Prod{n>0}E_n^*(B\Z/p^k) \lra{\cong} \Q \otimes \Prod{n>0}\,\Prod{0 \leq j \leq k} E_n^*\llbracket x\rrbracket/\langle p^j\rangle(x).
\]
\end{corollary}
\begin{proof}
We will prove the easiest case for clarity. Fix a height $n$ and let $k=1$, then with a coordinate this is the map
\[
\E^*\powser{x}/[p](x) \lra{} \E^* \times \E^*\powser{x}/\langle p \rangle (x)
\]
sending $x \mapsto (0,x)$ and where $[p](x)$ now indicates the $p$-series for the formal group law associated to $\E$. 
We may write down a basis for each side. For the left hand side we have the basis $\{1, x, \ldots, x^{p^n-1}\}$ and for the right hand side we can take the basis
\[
\{(1,1),(0,1),(0,x), \ldots, (0,x^{p^n-2})\}.
\]
It is clear that the basis elements $\{(1,1),(0,x),\ldots, (0,x^{p^n-2})\}$ are hit under this map. Dividing by a power of $p$ is required to hit the basis element $(0,1)$. But $BP$ provides a global element that hits $(0,1)$ at each height $n$. It is the element (using the notation of the proof of the previous proposition)
\[
-f(x)/p.
\]
Thus we see that we only need to divide by $p$ once in order to establish an isomorphism for all $n$. For $k>1$ the proof is similar, but we need to divide by $p^k$. For instance, when $k=2$
\begin{align*}
-f([p](x)) &\mapsto (0,0,p) \\
-f(x) &\mapsto (0,p,g(x)), 
\end{align*}
where $g(x)$ is some polynomial. This implies that
\[
-\frac{f(x)}{p} + \frac{f([p](x))g(x)}{p^2} \mapsto (0,1,0).
\]
Finally we see that the cokernel of the map 
\[
\Prod{n>0}E_n^*(B\Z/p^k) \lra{} \Prod{n>0}\,\Prod{0 \leq j \leq k} E_n^*\llbracket x\rrbracket/\langle p^j\rangle(x)
\]
is all $p^k$-torsion independent of the height. Thus the map is a rational isomorphism. 
\end{proof}

To prove a similar result for $\Sigma_p$, we need to recall a few basic results regarding $\E^*(B\Sigma_p)$. Let $I \subset \E^*(B\Sigma_p)$ be the image of the transfer along $e \subset \Sigma_p$. For each $n \in \N$, the natural map
\[
\E^*(B\Sigma_p) \lra{} \E^* \times \E^*(B\Sigma_p)/I
\]
is injective and induces an isomorphism
\[
\Q \otimes \E^*(B\Sigma_p) \lra{\cong} \Q \otimes(\E^* \times \E^*(B\Sigma_p)/I).
\]
Injectivity can be seen by considering Rezk's pullback square \cite[Proposition 10.5]{rezkcongruence}
\[
\xymatrix{\E^*(B\Sigma_p) \ar[r] \ar[d] & \E^*(B\Sigma_p)/I \ar[d] \\ \E^* \ar[r] & \E^*/p}
\]
and the isomorphism follows by applying $\Q \otimes -$ to the pullback square above using that $\Q$ is a flat $\Z$-module.

By considering the homotopy pullback of $Be \subset B\Sigma_p$ along $B\Z/p \rightarrow B\Sigma_p$ (the key point is that $|\Sigma_p/(\Z/p)|$ is coprime to $p$), one can see that there is a canonical map
\[
\E^*(B\Sigma_p)/I \rightarrow \E^*(B\Z/p)/I.
\]

\begin{lemma}
There is an isomorphism
\[
\Q \otimes \Prod{n}\E^*(B\Sigma_p) \lra{\cong} \Q \otimes \Prod{n} \E^* \times \E^*(B\Sigma_p)/I,
\]
where $I$ is the ideal generated by the transfer from the trivial group. 
\end{lemma}
\begin{proof}
Let $\Cok_{\E}(\Sigma_p)$ be the cokernel
\[\xymatrix{0 \ar[r] &  \E^*(B\Sigma_p) \ar[r] & \E^* \times \E^*(B\Sigma_p)/I \ar[r] & \Cok_{\E}(\Sigma_p) \ar[r] & 0.}\]
We will show that $\Cok_{\E}(\Sigma_p)$ is all $p$-torsion with exponent $1$, i.e., that it is annihilated by multiplication by $p$.

There is a stable splitting
\[\Sigma^{\infty}_+B\Z/p \simeq \Sigma^{\infty}_+B\Sigma_p \times Y\]
for $Y$ some spectrum. Applying $E$-theory to this yields a commutative diagram 
\[\xymatrix{E_n^*(B\Sigma_p) \ar@{^{(}->}[r] \ar@{^{(}->}[d]_i & E_n^* \times E_n^*(B\Sigma_p)/I \ar@{->>}[r] \ar@{^{(}->}[d] & \Cok_{\E}(\Sigma_p)\ar[d] \\
E_n^*{B\Z/p} \ar@{^{(}->}[r] & E_n^* \times E_n^*(B\Z/p)/I \ar@{->>}[r] & \Cok_{\E}(\Z/p)}\]
in which the first vertical map $i$ is split injective (by the stable splitting above). Since $E_n^*B\Z/p$ is free, the cokernel of $i$ is projective, hence $p$-torsion free. Since $\Q\otimes \Cok_{\E}(\Sigma_p)= 0$, the kernel of the map $\Cok_{\E}(\Sigma_p) \to \Cok_{\E}(\Z/p)$ must be torsion. Therefore the snake lemma implies that the kernel of the right vertical map $\Cok_{\E}(\Sigma_p) \to \Cok_{\E}(\Z/p)$ must be a torsion subgroup of a torsion free module and hence is zero. Thus $\Cok_{\E}(\Sigma_p)$ is a subgroup of $\Cok_{\E}(\Z/p)$. The claim now follows from \myref{cyclicdecomp}.
\end{proof}

\subsection{Commutative algebra} \mylabel{sec:4.2}
In this subsection we prove some basic facts about integer torsion in modules.

The following lemma is part of a well-known collection of facts that go under the title ``fracture squares".
\begin{lemma} \mylabel{pb}
Let $R = L_{K(t)}\E^0 \cong W(\kappa)\powser{u_1,\ldots, u_{n-1}}[u_{t}^{-1}]^{\wedge}_{I_t}$ for $t>0$ and let $S$ be a subset of $R - (p)$, then there is a pullback square of commutative rings
\[
\xymatrix{R \ar[r] \ar[d] & (S^{-1}R)^{\wedge}_{p} \ar[d] \\ \Q \otimes R \ar[r] & \Q \otimes (S^{-1}R)^{\wedge}_{p}.
}
\]
Moreover, for any finitely generated free $R$-module $M$ there is a similar square given by tensoring the pullback square with $M$.
\end{lemma}
\begin{proof}
The square in the proposition is the composite of two squares:
\[
\xymatrix{R \ar[r] \ar[d] & S^{-1}R \ar[r] \ar[d] & (S^{-1}R)^{\wedge}_{p} \ar[d] \\ \Q \otimes R \ar[r] & \Q \otimes S^{-1}R \ar[r] & \Q \otimes (S^{-1}R)^{\wedge}_{p}.
}
\]
The right-hand square is an ``arithmetic square". It is a pullback square by applying \cite[4.12, 4.13]{dwyergreenlees}. In \cite{dwyergreenlees}, they show that the derived functors of localization and completion fit into a homotopy pullback square. The result we want is recovered by noticing that $S^{-1}R$ is free as an $S^{-1}R$-module so $H_0(S^{-1}R) = S^{-1}R$ and that the derived completion and ordinary completion agree because $S^{-1}R$ is Noetherian.

The left-hand square is a pullback by direct computation. Let $\frac{r}{{s^k}} \in S^{-1}R$ and $\frac{r'}{p^l} \in \Q \otimes R$ such that 
\[
\frac{r}{{s^k}} = \frac{r'}{p^l} \in \Q \otimes S^{-1}R,
\]
where $r,r' \in R$. Thus $s^k r' = p^l r$ and 
\[
r \mapsto \frac{s^k r'}{p^l} \in \Q \otimes R.
\]
But this implies that $\frac{s^k r'}{p^l} \in R$ since the map is an inclusion. Since $R$ is an integral domain and $s$ is, by definition, not in the ideal generated by $p$,
\[
p^l | r' \in R.
\]
Now $r'/p^l$ is the pullback.

\end{proof}

\begin{lemma} \mylabel{free}
Let $M \lra{f} N$ be an injective map of finitely generated free $R$-modules. If $x \in N$ has the property that $p^lx \in \im f$ and $s^kx \in \im f$ for $s \in S$, then $x \in \im f$.
\end{lemma}
\begin{proof}
It suffices to set $k = l =1$. We will use \myref{pb}. There are $m_1, m_2 \in M$ such that $\frac{m_1}{s} \in S^{-1}M^{\wedge}_p$ maps to $x$ and $\frac{m_2}{p} \in p^{-1}M$ maps to $x$. Now consider the commutative cube
\[
\xymatrix{& M \ar[ld]_{f} \ar'[d][dd] \ar[rr] & & S^{-1}M^{\wedge}_p \ar[ld]^{} \ar[dd] \\
N \ar[dd] \ar[rr] & &  S^{-1}N^{\wedge}_p \ar[dd] \\
& p^{-1}M \ar[ld]_{} \ar'[r][rr] & &  p^{-1}(S^{-1}M^{\wedge}_p) \ar[ld]^{} \\
p^{-1}N \ar[rr]  &&  p^{-1}(S^{-1}N^{\wedge}_p).}
\]
The front and back faces are pullback squares. The elements $\frac{m_1}{s}$ and $\frac{m_2}{p}$ must agree in $p^{-1}(S^{-1}M^{\wedge}_p)$ and thus there is an element $m \in M$ that maps to both of these. By commutativity we must have that $f(m) = x$.
\end{proof}

By the torsion of an $R$-module $M$ we will always mean the set of elements $x \in M$ such that $nx = 0$ for some nonzero $n \in \N$.

\begin{definition}
For an abelian group $M$ let $e(M)$ be the exponent of the torsion in $M$. That is, let $K$ be the kernel of $M \lra{} \Q \otimes M$, then $e(M)$ is the minimum number $n \in \N \cup \{\infty\}$ such that $nK=0$.
\end{definition}

\begin{proposition}
Let $R$ be a Noetherian ring and $M$ be a finitely generated $R$-module, then the torsion in $M$ is bounded.
\end{proposition}
\begin{proof}
Let $K$ be the kernel
\[
0 \lra{} K \lra{} M \lra{} \Q \otimes M = (\Q \otimes R) \otimes_R M.
\]
The $R$-module $K$ is finitely generated since $R$ is Noetherian and $M$ is finitely generated. Choose generators $k_1, \ldots, k_m$ of $K$; there exists an $n$ such that $nk_i = 0 $ for all $i$. Now any element of $K$ is of the form $\sum r_i k_i$ and this is also killed by $n$.
\end{proof}

\begin{proposition} \mylabel{diagonal}
Let $R$ be a torsion-free Noetherian ring. Let $f \colon N \hookrightarrow M$ be an injection of finitely generated free $R$-modules and let $\Delta:M \hookrightarrow M \times M$ be the diagonal, then $e(\Cok(f)) = e(\Cok(\Delta f))$.
\end{proposition}
\begin{proof}
Consider the map of short exact sequences
\[
\xymatrix{N \ar[r]^{f} \ar[d]^{=} & M \ar[r] \ar[d]^{\Delta} & \Cok(f) \ar[d]^{g} \\ N \ar[r] & M \times M \ar[r] & \Cok(\Delta f).} 
\]
The diagonal $\Delta$ is a section of the projection onto a factor and the left vertical arrow is an equality so $g$ is an injection. Thus there there is an induced map $\Cok(\Delta f) \lra{r} \Cok(f)$ such that $rg = 1$. Since integer torsion is sent to integer torsion, this implies that the map induced by $g$ on integer torsion is a split injection. It is an isomorphism since a torsion element must be of the form $(m,m) \in M \times M$ with $(nm,nm) \in \Delta f N$ and thus the kernel of $r$ is trivial.
\end{proof}

\begin{proposition} \mylabel{extension}
Let $R$ be Noetherian and torsion free. For a flat $R$-algebra $T$ and a finitely generated $R$-module $M$, $e(M) \geq e(T\otimes_R M)$. If $T$ is a faithfully flat $R$-algebra then there is an equality $e(M) = e(T \otimes_R M)$.
\end{proposition}
\begin{proof}
Assume that $e(M) =n$ and let $[n]\colon M \lra{} M$ be the multiplication by $n$ map. Consider the exact sequence
\[
0 \lra{} K \lra{} M \lra{[n]} \im [n] \lra{} 0,
\]
where $K$ is the torsion in $M$. Base change to $T$ gives a short exact sequence and $T \otimes_R \im [n]$ is necessarily torsion free, thus 
\[
e(T \otimes_R K) = e(T \otimes_R M) \leq n. 
\]
If $T$ is faithfully flat then $K \lra{} T \otimes_R K$ is injective (\cite[4.74,4.83]{lambook}), so $e(T \otimes_R M) = e(M)$.
\end{proof}

\begin{corollary} \mylabel{locflat}
Let $R$ be as in \myref{pb}, let $T$ be a faithfully flat $(S^{-1}R)^{\wedge}_{p}$-algebra, and let $M \lra{f} N$ be a map of finitely generated free $R$-modules, then 
\[
e(N/\im f) = e(T \otimes_{R} (N/ \im f)).
\]
\end{corollary}
\begin{proof}
By \myref{free} the torsion in $N/ \im f$ is contained in the torsion of $(S^{-1}R)^{\wedge}_{p} \otimes_R N/ \im f$, so
\[
e(N/\im f) \leq e((S^{-1}R)^{\wedge}_{p} \otimes_R N/ \im f).
\]
Also $(S^{-1}R)^{\wedge}_{p}$ is a flat extension of $R$ and this gives the reverse inequality. Since $T$ is faithfully flat, base change to $T$ is injective. Thus we have have equalities
\[
e(N/ \im f) = e((S^{-1}R)^{\wedge}_{p} \otimes_R N/ \im f) = e(T \otimes_{R} (N/ \im f))
\]
by \myref{extension}.
\end{proof}

\subsection{A few hand-crafted cohomology theories} \mylabel{sec:4.3}
In order to prove the bounded torsion results for $E$-theory we make use of several $\E$-algebras. In this section we construct the $\E$-algebras that we need and prove or recall some basic properties of them. This relies heavily on ideas from \cite{bscentralizers}.

Let $L_{1} = L_{K(1)}\E$ (it does depend on $n$) and recall from \cite[Cor.~1.5.5]{Hovey-vn} that
\[
\pi_0 L_{1} = W(\kappa)\powser{u_1, \ldots, u_{n-1}}[u_{1}^{-1}]^{\wedge}_p.
\]
This is a flat $\E^0$-algebra. Recall that we may localize an $E_{\infty}$-ring spectrum $R$ at a prime ideal $\mathfrak{p} \subseteq \pi_0R$ by inverting every element in $\pi_0R - \mathfrak{p}$; we write $R_{\mathfrak{p}}$ for the localized ring spectrum. Define
\[
F_1 = L_{K(1)} ((\E)_{(p)});
\]
by inverting all of the elements away from $\mathfrak{p} = (p)$ and then $K(1)$-localizing the resulting spectrum. Note that this inverts $u_1, \ldots, u_{n-1}$. The ring spectrum $F_1$ is an even-periodic Landweber-exact $E_{\infty}$-ring such that 
\[
\pi_0F_1 = F_{1}^{0} = ((\E^0)_{(p)})^{\wedge}_{p}.
\]
Since $F_1$ is $K(1)$-local there is a canonical map
\[
L_{1} \lra{} F_1
\]
inducing the obvious map on coefficients. This suggests another construction of $F_1$. It is clear that
\[
((L_{1})_{(p)})^{\wedge}_{p} \lra{\cong} F_1
\]
by considering the quotient of this ring by powers of $p$. This allows \myref{locflat} to be used with $(R, S) = (\E^*, \E^* - (p))$ or $(R,S) = (L_{1}^*, L_{1}^*-(p))$.  

Now let $G$ be a finite group.

\begin{proposition} \mylabel{Fequiv}
The canonical map 
\[
F_1 \wedge_{L_1} L_{1}^{BG} \lra{\simeq} F_{1}^{BG} 
\]
is an equivalence natural in $G$.
\end{proposition}
\begin{proof}
The proof follows the proof of Proposition 6.2 in \cite{bscentralizers}. It is clear that the map is natural. To prove it is an equivalence, the main idea is to note that the smash product one the left is $K(1)$-local and then use the $K(1)$-local self-duality of classifying spaces of finite groups.
\end{proof}

Following ideas in \cite{bscentralizers} based on work of Hopkins in \cite{coctalos}, we further define
\[
\bar{F}_1 = L_{K(1)}(F_1 \wedge K_p),
\]
where $K_p$ is $p$-adic $K$-theory. We do not have much of control over the coefficients of this ring spectrum. However, following Proposition 5.8 of \cite{bscentralizers}, $\bar{F}_1$ does have the following very desirable property:

\begin{proposition} \mylabel{faithful}
The ring of coefficients of $\bar{F}_1$ is faithfully flat as a $K_{p}^*$-module and as an $F_{1}^*$-module.
\end{proposition}
\begin{proof}
This follows from the proof of Proposition 5.8 in \cite{bscentralizers} which relies on the important (and somewhat surprising) Proposition 3.13 which extends an argument of Hovey's in \cite{hoveynotes}. Both module structures are faithfully flat because $K_{p}^*$ and $F_{1}^*$ are both complete local Noetherian. 
\end{proof}

\begin{proposition} \mylabel{equivs}
For a finite group $G$, there are natural equivalences
\[
\bar{F}_1 \wedge_{F_1} F_{1}^{BG} \lra{\simeq} \bar{F}_{1}^{BG} \lla{\simeq} \bar{F}_1 \wedge_{K_p} K_{p}^{BG}.
\]
\end{proposition}
\begin{proof}
This also follows the proof of Proposition 6.2 in \cite{bscentralizers}, using \myref{faithful}.
\end{proof}

Now let $H \subset G$ so that we have a transfer map $E^{BH} \lra{\text{Tr}} E^{BG}$ for any cohomology theory $E$. 

\begin{proposition}
Let $E$ and $F$ be $E_{\infty}$-rings and let $F$ be an $E$-algebra. There is a commutative square
\[
\xymatrix{F \wedge_E E^{BH} \ar[r] \ar[d]_{F \wedge_E \text{Tr}} & F^{BH} \ar[d]^{\text{Tr}} \\ F \wedge_E E^{BG} \ar[r] & F^{BG}.}
\]
\end{proposition}
\begin{proof}
The map of $E_{\infty}$-rings $E \lra{} F$ and the transfer map $\Sigma^{\infty}_{+} BH \lra{} \Sigma^{\infty}_{+} BG$  induce a commutative square of function spectra
\[
\xymatrix{E^{BH} \ar[r] \ar[d]_{\text{Tr}} & F^{BH} \ar[d]^{\text{Tr}} \\ E^{BG} \ar[r] & F^{BG}}
\]
and now we may base change the left hand side to $F$ over $E$.
\end{proof}

Let $E$ be any cohomology theory, let $H_1, \ldots, H_m \subset G$, and let $I \subset E^*(BG)$ be the ideal generated by the image of the transfers $E^*(BH_i) \lra{\text{Tr}} E^*(BG)$. This is the image of the map
\[
\Oplus{i} E^*(BH_i) \lra{\oplus \text{Tr}} E^*(BG).
\]

\begin{proposition} \mylabel{transferiso}
Let $E = K_p, L_{1}, F_1$ or $\bar{F}_1$, let $H_1, \ldots, H_m \subset G$, and let $I \subset E^*(BG)$ be the associated transfer ideal. There are isomorphisms
\[
F_{1}^* \otimes_{L_{1}^*} L_{1}^*(BG)/I \lra{\cong} F_{1}^*(BG)/I
\]
and
\[
\bar{F}_{1}^* \otimes_{F_{1}^*} F_{1}^*(BG)/I \lra{\cong} \bar{F}_{1}^*(BG)/I \lla{\cong} \bar{F}_{1}^* \otimes_{K_{p}^*} K_{p}^*(BG)/I.
\]
\end{proposition}
\begin{proof}
The first isomorphism follows from the previous proposition and \myref{Fequiv}, and the second isomorphism follows from the previous proposition and \myref{equivs}.
\end{proof}

In the next section we will make use of character theory for Morava $E$-theory. This will make use of two more $\E$-algebras. The first was introduced in \cite{hkr} and the second was introduced in \cite{tgcm}.

Let $\Lambda_{k,n} = (\Z/p^k)^n$. The $\E$-algebra $C_0$ is constructed in Section 6.2 of \cite{hkr}. It is a $p^{-1}\E$-algebra and may be constructed as an $\E$-algebra as a localization of the colimit
\[
\Colim{k} \text{ } \E^{B\Lambda_{k,n}}.
\]
The commutative ring $C_{0}^*$ is faithfully flat as a $\Q \otimes \E^*$-algebra. The point of the construction (and this is Theorem C of \cite{hkr}) is that, for any finite group $G$, there is a canonical equivalence
\[
C_0 \wedge_{\E} \E^{BG} \lra{\sim} C_{0} \wedge_{p^{-1}\E} (p^{-1}\E)^{\Loops^nBG},
\]
where $\Loops(-) = \hom(B\Z_p,-)$ is the $p$-adic free loop space functor. Since $p^{-1}\E$ is rational and $C_{0}^*$ is faithfully flat as a $p^{-1}\E^*$-algebra, the target is equivalent to
\[
\Prod{\pi_0\Loops^nBG}C_0.
\]

The $\E$-algebra $C_1$ is built (as a commutative ring) in \cite{tgcm} and as an $\E$-algebra in \cite{bscentralizers}. It is a $K(1)$-local $L_1$-algebra which is given as the $p$-completion of a localization of the colimit
\[
\Colim{k} \text{ } L_1 \wedge_{\E} \E^{B\Lambda_{k,n-1}}.
\]
The point of the construction (and this is a special case of the main theorem of \cite{tgcm}) is that there is an equivalence
\[
C_1 \wedge_{\E} \E^{BG} \lra{\sim} C_{1}\wedge_{L_1}L_{1}^{\Loops^{n-1}BG}.
\]
Since $C_{1}^*$ is flat as an $\E^*$-algebra and faithfully flat as an $L_{1}^*$-algebra, applying $\pi_*$ to this equivalence gives the isomorphism
\begin{equation} \label{ht1}
C_{1}^* \otimes_{\E^*} \E^*(BG) \lra{\cong} C_{1}^* \otimes_{L_{1}^*} L_{1}^*(\Loops^{n-1}BG).
\end{equation}
The iterated $p$-adic free loop space of $BG$ is a disjoint union of classifying spaces of certain centralizer subgroups of $G$ and this is how the codomain is most easily understood explicitly. Let 
\[
\hom(\Z_{p}^{l},G)/\sim
\]
be conjugacy classes of continuous maps from $\Z_{p}^l$ to $G$, then
\[
C_{1}^* \otimes_{L_{1}^*} L_{1}^*(\Loops^{n-1}BG) \cong \Prod{[\al] \in \hom(\Z_{p}^{n-1},G)/\sim} C_{1}^* \otimes_{L_{1}^*} L_{1}^*(\Loops^{n-1}BC(\im \al)),
\]
where $C(\im \al)$ is the centralizer of the image of a choice $\al \in [\al]$.

The equivalences above behave well with respect to transfers. Formulas can be found in Theorem D of \cite{hkr} and \cite[Theorem 2.18, Theorem 3.11]{subpdiv}. 

\subsection{Bounded torsion in $E$-theory}
In this subsection we prove that rational factorizations of the Morava $E$-cohomology of finite abelian groups and symmetric groups hold integrally up to a torsion cokernel with a height-independent bounded exponent. We do this by putting together the tools developed in the last two subsections.

When we write $I \subset E^*(BA)$ for an abelian group $A$ we will always mean the ideal generated by the image of the transfer along proper subgroups of $A$. When we write $I \subset E^*(B\Sigma_{m})$ we will always mean the ideal generated by the image of the transfer along the proper partition subgroups $\Sigma_i \times \Sigma_j \lra{} \Sigma_{m}$ where $i+j = m$ and $i,j>0$.  

We will begin with symmetric groups. First we will set up the algebro-geometric objects that witness the rational factorization.

Let $\lambda \vdash m$ be a partition of $m$. We write 
\[
\lambda = a_1 \lambda_1  + a_2\lambda_2 + \ldots +  a_k\lambda_k, 
\]
where $\lambda_i < \lambda_j$ when $i<j$, $a_i \in \N$, and the value of the sum is $m$. To ease notation, we write $\Sigma_{\lambda} = \Sigma_{\lambda_1}^{\times a_1} \times \ldots \times \Sigma_{\lambda_k}^{\times a_k}$ and $\Sigma_{\underline{a}} = \Sigma_{a_1} \times \ldots \times \Sigma_{a_k}$. Furthermore, let $I_{\lambda}$ be the ideal in $\E^*(B\Sigma_{\lambda})$ generated by the individual transfer ideals $I \subset \E^*(B\Sigma_{\lambda_i})$. That is, $I_{\lambda}$ has the property that 
\[
\E^*(B\Sigma_{\lambda})/I_{\lambda} \cong \Otimes{i} (\E^*(B\Sigma_{\lambda_i})/I)^{\otimes a_i}.
\]
Furthermore, there is an action of $\Sigma_{a_i}$ on the tensor power
\[
(\E^*(B\Sigma_{\lambda_i})/I)^{\otimes a_i}.
\]
For each tensor factor, we will take the fixed points by this action and denote the result by
\[
(\E^*(B\Sigma_{\lambda})/I_{\lambda})^{\Sigma_{\underline{a}}} = \Otimes{i} ((\E^*(B\Sigma_{\lambda_i})/I)^{\otimes a_i})^{\Sigma_{a_i}}.
\]

This commutative ring can be understood algebro-geometrically by a theorem of Strickland's. In \cite{subgroups}, Strickland proves that $\E^*(B\Sigma_{m})/I$ is a finitely generated free $\E^*$-module and produces a canonical isomorphism
\[
\Spf(\E^*(B\Sigma_{m})/I) \cong \Sub_m(\G_{\E}),
\]
where $\Sub_m(\G_{\E})$ is the scheme classifying subgroup schemes of order $m$ in $\G_{\E}$. It is worth noting that this is the empty scheme if $m$ is coprime to $p$. The scheme associated to the ring that we are interested in is built out of $\Sub_m(\G_{\E})$. Thus we may define
\begin{align*}
\Sub_{\lambda \vdash m}(\G_{\E}) & = \Spec((\E^*(B\Sigma_{\lambda})/I_{\lambda})^{\Sigma_{\underline{a}}}) \\
& \cong \Prod{i}\Spec(((\E^*(B\Sigma_{\lambda_i})/I)^{\otimes a_i})^{\Sigma_{a_i}}).
\end{align*}
The point of the notation $\Sub_{\lambda \vdash m}(\G_{\E})$ is that this scheme represents unordered sets of $\sum_i a_i$ subgroups of $\G_{\E}$ in which $a_i$ of the subgroups have order $\lambda_i$. This scheme theoretic description is not essential to the proof of the result. However, one of the goals of decomposing $BP^*(B\Sigma_m)$ rationally is to suggest that it may have a tractable algebro-geometric description. 

Let $\Gamma(-)$ be the functor that takes a scheme to the ring of functions on the scheme. There is a canonical map
\[
\E^*(B\Sigma_{m}) \lra{} \Gamma \Sub_{\lambda \vdash m}(\G_{\E})
\]
induced by restriction. The only part to check here is that it actually lands in the fixed points, but this is clear as there is an inner automorphism of $\Sigma_m$ that permutes the $a_i$ groups of the form $\Sigma_{\lambda_i}$ sitting inside of $\Sigma_m$.

\begin{proposition}
The canonical map
\[
\E^*(B\Sigma_{m}) \lra{} \Prod{\lambda \vdash m}\Gamma \Sub_{\lambda \vdash m}(\G_{\E})
\]
is injective and a rational isomorphism.
\end{proposition}
\begin{proof}
This is a basic application of HKR-character theory \cite{hkr}. It is clear that the map is an isomorphism after base change to $C_{0}^*$. The resulting ring is the $C_{0}^*$-valued functions on the set
\[
\Sum_m(\QZ{n}) = \Coprod{\lambda \vdash m} \Sub_{\lambda \vdash m}(\QZ{n}).
\]
But $C_{0}^*$ is a faithfully flat $p^{-1}\E^*$-algebra, so the map is a rational isomorphism.
\end{proof}

Most of the results of Sections \ref{sec:4.2} and \ref{sec:4.3} will go into the proof of the next theorem. We will use the method of the proof twice more in \myref{leveliso} and \myref{artin}. The idea of the proof is the following: We will begin with a short exact sequence with a torsion cokernel. We will apply character theory to height $1$ and use faithful flatness to get a short exact sequence in $L_1$-cohomology. In this case, the cokernel will be a product of cokernels of maps. The sorts of maps that can occur are bounded by the number of centralizers of tuples of commuting elements in $\Sigma_m$. However, $F_1$ depends on $n$, so we do not have complete control over the cokernel as $n$ varies. We will extend to $\bar{F}_1$ and then use faithful flatness to change the product of maps to a product of maps in $K_p$-cohomology. Now we are in a situation that is height independent, the number of maps that can appear in the product is bounded, and the cokernel of each map is bounded torsion.

\begin{theorem} \mylabel{symbound}
The cokernel of the map
\[
\E^*(B\Sigma_{m}) \lra{} \Prod{\lambda \vdash m}\Gamma \Sub_{\lambda \vdash m}(\G_{\E})
\]
is torsion with exponent bounded independent of the height $n$.
\end{theorem}
\begin{proof}
Note that both the source and the target are finitely generated free $\E^*$-modules and the map is an injection. The cokernel $\Cok_{\E}(\Sigma_{m})$ is a torsion $\E^*$-module. We will repeatedly base change the short exact sequence
\[
0 \lra{} \E^*(B\Sigma_m) \lra{} \Prod{\lambda \vdash m}\Gamma \Sub_{\lambda \vdash m}(\G_{\E}) \lra{} \Cok_{\E}(\Sigma_m) \lra{} 0
\]
to prove the result. Base change to $C_{1}^{*}$ (a flat $\E^*$-algebra) and Equation \ref{ht1} give the short exact sequence
\[\resizebox{\columnwidth}{!}{
\xymatrix{\Prod{[\beta] \in \Hom(\Z_{p}^{n-1},\Sigma_{m})/\sim} C_{1}^* \otimes_{L_{1}^*} L_{1}^*(BC(\im \beta)) \ar[r] & \Prod{\lambda \vdash m}\Gamma \Sub_{\lambda \vdash m}(\G_{C_1} \oplus \QZ{n-1}) \ar[r] & C_{1}^* \otimes_{L_{1}^*} L_{1}^* \otimes_{\E^*} \Cok_{\E}(\Sigma_m).}}
\]
We have used that the scheme classifying subgroups behaves well under base change (Theorem 3.11 in \cite{subpdiv} is particularly relevant). \myref{locflat} implies that the exponent of the torsion cannot have changed. By faithful flatness (\myref{extension}), we may remove the $C_{1}^*$ without changing the exponent of the torsion. Thus we arrive at the short exact sequence
\begin{equation}\label{eq:ses}
\Prod{[\beta]} L_{1}^*(BC(\im \beta)) \lra{} \Prod{\lambda \vdash m}\Gamma \Sub_{\lambda \vdash m}(\G_{L_1} \oplus \QZ{n-1}) \lra{}  L_{1}^* \otimes_{\E^*} \Cok_{\E}(\Sigma_m).
\end{equation}
By the fact that the first and second terms are finitely generated and free, we may apply \myref{locflat} to base change \eqref{eq:ses} to $F_{1}^*$ without changing the exponent of $L_{1}^* \otimes_{\E^*} \Cok_{\E}(A)$. We may then further base change to $\bar{F}_{1}^*$ (which is faithfully flat over $F_{1}^*$ by \myref{faithful}) giving
\[
\Prod{[\beta]} \bar{F}_{1}^*(BC(\im \beta)) \lra{} \Prod{\lambda \vdash m}\Gamma \Sub_{\lambda \vdash m}(\G_{\bar{F}_1} \oplus \QZ{n-1}) \lra{}  \bar{F}_{1}^* \otimes_{\E^*} \Cok_{\E}(\Sigma_m).
\]
Now \myref{equivs} gives the short exact sequence
\[
\Prod{[\beta]}\bar{F}_{1}^* \otimes_{K_{p}^*} K_{p}^*(BC(\im \beta)) \lra{} \Prod{\lambda \vdash m}\Gamma \Sub_{\lambda \vdash m}(\G_{\bar{F}_1} \oplus \QZ{n-1}) \lra{}  \bar{F}_{1}^* \otimes_{\E^*} \Cok_{\E}(\Sigma_m).
\]
By the naturality of the equivalence, the first map in the sequence is induced by restriction maps. Let 
\[
M_{A,n} = \Cok(\Prod{[\beta]} K_{p}^*(BC(\im \beta)) \lra{} \Prod{\lambda \vdash m}\Gamma \Sub_{\lambda \vdash m}(\G_{m} \oplus \QZ{n-1})),
\]
then 
\[
\bar{F}_{1}^* \otimes_{K_{p}^*} M_{A,n} \cong \bar{F}_{1}^* \otimes_{\E^*} \Cok_{\E}(\Sigma_m).
\]
By \myref{extension}, it suffices to show that the torsion in $M_{A,n}$ is independent of $n$. The map
\[
\Prod{[\beta]} K_{p}^*(BC(\im \beta)) \lra{} \Prod{\lambda \vdash m}\Gamma \Sub_{\lambda \vdash m}(\G_{m} \oplus \QZ{n-1})
\]
is the ring of functions on the canonical map
\[
\Coprod{\lambda \vdash m} \Sub_{\lambda \vdash m}(\G_m \oplus \QZ{n-1}) \lra{f} \Spec(K_{p}^*(\Loops^{n-1} B\Sigma_m)).
\]
There is a commutative triangle
\[
\xymatrix{\Coprod{\lambda \vdash m} \Sub_{\lambda \vdash m}(\G_m \oplus \QZ{n-1}) \ar[r] \ar[dr] & \Spec(K_{p}^*(\Loops^{n-1} B\Sigma_m)) \ar[d] \\ & \Sum_{\leq m}(\QZ{n-1}),}
\]
where 
\[
\Sum_{\leq m}(\QZ{n-1}) = \{\oplus H_i | H_i \subset \QZ{n-1} \text{ and } \sum_i|H_i| \leq m\}
\]
is the set of formal sums of subgroups in $\QZ{n-1}$ with order $\leq m$. For each element $\alpha \in \Sum_{\leq m}(\QZ{n-1})$, the map that lives over $\alpha$ is a rational isomorphism (and an injection after applying $\Gamma(-)$). Thus it suffices to prove that, regardless of the choice of $n$, there are only finitely many different maps that can appear as fibers. This is essentially a consequence of the fact that there are only finitely many groups that are centralizers of $\Sigma_m$. 

To make this precise, note that there is an action of $\Aut(\QZ{n-1})$ on the triangle. This action must send the map over $\alpha \in \Sum_{\leq m}(\QZ{n-1})$ to a map with the same exponent of the torsion in the cokernel. The action of $\Aut(\QZ{n-1})$ is compatible with the inclusion of triangles induced by the inclusion of $\QZ{n-1} \lra{} \QZ{n}$ as the first $n-1$ summands. For $n$ large enough any sum of subgroups of size $\leq m$ in $\QZ{n}$ is isomorphic to a sum of subgroups in $\QZ{n-1}$ under this action, and this finishes the proof.
\end{proof}

It is worth noting that this could probably be factored further using Strickland's schemes called $\text{Type}$ and Theorem 12.4 in \cite{subgroups}.

Now we turn our attention to an analogous theorem for level structures. The theory of level structures for a formal group is developed in \cite{subgroups}. We make use of it now because there is a close relationship between $\E^*(BH)/I$ and $\Gamma \Level(H^*, \G_{\E})$, where $H^*$ is the Pontryagin dual of $H$. In \cite{ahs}, it is proved that the image of $\E^*(BH)/I$ in $\Q \otimes \E^*(BH)/I$ is canonically isomorphic to $\Gamma \Level(H^*, \G_{\E})$.

\begin{theorem} \mylabel{leveliso}
For $A$ a finite abelian group, the cokernel of the map
\[
\E^*(BA) \lra{} \Prod{H \subseteq A} \Gamma \Level(H^*, \G_{\E}) 
\]
is torsion with exponent bounded independent of the height $n$.
\end{theorem}
\begin{proof}
Since the argument is similar to the proof of \myref{symbound}, we only point out the differences. We can again reduce to showing that the torsion in 
\[
M_{A,n} = \Cok(\Prod{\beta} K_{p}^*(BA) \lra{} \Prod{H \subseteq A} \Gamma \Level(H^*, \G_{m} \oplus \QZ{n-1}))
\]
has exponent independent of $n$. To this end, note that the map
\[
\Prod{\beta} K_{p}^*(BA) \lra{} \Prod{H \subseteq A} \Gamma \Level(H^*, \G_{m} \oplus \QZ{n-1})
\]
is obtained by taking global sections on the canonical map
\[
\Coprod{H \subset A}\Level(H^*, \G_{m} \oplus \QZ{n-1}) \lra{f} \Hom(A^*,\G_{m} \oplus \QZ{n-1}).
\]
It is a product of ring maps since we have the following commutative triangle
\[
\xymatrix{\Coprod{H \subset A}\Level(H^*, \G_{m} \oplus \QZ{n-1}) \ar[r] \ar[dr] & \Hom(A^*,\G_{m} \oplus \QZ{n-1}) \ar[d] \\ & \Hom(A^*, \QZ{n-1})}
\]
and $\Hom(A^*, \QZ{n-1}) \cong \Hom(\Z_{p}^{n-1},A)$ is a set and thus disconnects the horizontal map. The diagonal arrow sends a level structure $H^* \hookrightarrow \G_{m} \oplus \QZ{n-1}$ to the composite
\[
A^* \twoheadrightarrow H^* \hookrightarrow \G_{m} \oplus \QZ{n-1} \twoheadrightarrow \QZ{n-1}.
\]
Let $f_{\beta}$ be the fiber over $\beta$. It is a map of schemes. It is clear that there is an action of $\Aut(\QZ{n-1})$ on the triangle above. This implies that the fibers over maps $\beta, \beta' \in \Hom(\Z_{p}^{n-1},A)$ with $\im \beta = \im \beta'$ are isomorphic. Thus the exponent of the cokernel of $\Gamma(f_{\beta})$ equals the exponent of the cokernel of $\Gamma(f_{\beta'})$. Adding another $\QZ{}$ to the constant \'etale part does not affect this in the sense that the induced inclusion
\[
\Hom(A^*, \QZ{n-1}) \lra{} \Hom(A^*, \QZ{n})
\]
lifts to an inclusion of triangles. As $n$ increases above the number of generators of $A$, $e(M_{A,n})$ remains constant. The fibers over all of the maps in $\Hom(A^*, \QZ{n})$ are isomorphic to the fibers over maps that were already there in $\Hom(A^*, \QZ{n-1})$.  Thus the exponent of the torsion in $M_{A,n}$ is bounded independent of $n$.
\end{proof}

The proof of the following result is given in the appendix, due to Jeremy Hahn.

\begin{proposition}\mylabel{transfertorsion}
Let $H$ be a finite abelian group. The exponent of the torsion in 
\[
\E^*(BH)/I
\]
is bounded independent of the height $n$.
\end{proposition}
 
\begin{corollary}\mylabel{cor:bpabdcomp}
If $A$ is a finite abelian group, then the map
\[
\Prod{n} \E^*(BA) \lra{} \Prod{n} \Prod{H \subseteq A} \E^*(BH)/I
\]
is a rational isomorphism.
\end{corollary}
\begin{proof}
Consider the composite
\[
\Prod{n} \E^*(BA) \lra{} \Prod{n} \Prod{H \subseteq A} \E^*(BH)/I \lra{} \Prod{n}\Prod{H \subseteq A} \Gamma \Level(H^*, \G_{\E}).
\]
By \myref{transfertorsion} the second map is a rational isomorphism and by \myref{leveliso} the composite 
is a rational isomorphism. Thus the first map is a rational isomorphism.
\end{proof}

\subsection{$BP$-cohomology of abelian and symmetric groups}

We show that the rational decomposition of \myref{symbound} extends to a rational decomposition of $BP$-cohomology.

\begin{lemma}\mylabel{lem:bpretracttr}
Let $G$ be a finite group and let $H_1, \ldots, H_m \subset G$ be subgroups. Let $I$ be the ideal generated by the image of the transfer maps along the inclusions $H_i \subset G$. There is a split inclusion
\[\xymatrix{BP^*(BG)/I \ar[r] & \Prod{n>0}E_n^*(BG)/I}\]
which is compatible with the map of \myref{cor:bpretract}.
\end{lemma}
\begin{proof}
Let $H_1, \ldots H_m \subseteq G$ be subgroups and consider the commutative diagram
\[\xymatrix{\Oplus{i}BP^*(BH_i) \ar[r]^{\mathrm{tr}} \ar@<1ex>[d] & BP^*(BG) \ar[r] \ar@<1ex>[d] & BP^*(BG)/I \ar[r] \ar@<1ex>[d] & 0 \\
\prod_{n>0}\Oplus{i}E_n^*(BH_i) \ar[r]_{\mathrm{tr}} \ar@/^0.5pc/@{-->}[u] & \prod_{n>0}E_n^*(BG) \ar[r]  \ar@/^0.5pc/@{-->}[u] & \prod_{n>0}E_n^*(BG)/I \ar[r] \ar@/^0.5pc/@{..>}[u] & 0,}\]
where $I$ denotes the transfer ideal generated by $H_1, \ldots, H_m$. The dashed arrows are the retracts which exist by \myref{cor:bpretract}, and the transfers make the left square commute by naturality and the fact that the splitting exists on the level of spectra. Therefore, the maps between the quotients exist and the resulting retract is compatible with the rest of the diagram.
\end{proof}

Consequently, for any symmetric group $\Sigma_m$, we have a commutative diagram
\[\xymatrix{\Q \otimes BP^*(B\Sigma_m) \ar@{^{(}->}[r] \ar@<1ex>[d] & \Q \otimes \Prod{\lambda \vdash m} \Otimes{i} ((BP^*(B\Sigma_{\lambda_i})/I)^{\otimes a_i})^{\Sigma_{a_i}} \ar@<1ex>[d] \\
\Q \otimes \Prod{n>0}E_n^*(B\Sigma_m) \ar@{^{(}->}[r] \ar@{^{(}->}[d] \ar@/^0.5pc/@{-->}[u] & \Q \otimes \Prod{n>0}\Prod{\lambda \vdash m} \Otimes{i} ((\E^*(B\Sigma_{\lambda_i})/I)^{\otimes a_i})^{\Sigma_{a_i}} \ar[d] \ar@/^0.5pc/@{-->}[u] \\
\Prod{n>0}\Q \otimes E_n^*(B\Sigma_m) \ar[r]^-{\cong} & \Prod{n>0}\Q \otimes \Prod{\lambda \vdash m} \Otimes{i} ((\E^*(B\Sigma_{\lambda_i})/I)^{\otimes a_i})^{\Sigma_{a_i}}.}\]
In \myref{symbound}, we showed that the middle horizontal map is an isomorphism. Now it is an easy consequence that the top map is an isomorphism.

\begin{corollary} \mylabel{symcor}
There is a rational isomorphism
\[
BP^*(B\Sigma_{m}) \lra{} \Prod{\lambda \vdash m} (BP^*(B\Sigma_{\lambda})/I_{\lambda})^{\Sigma_{\underline{a}}}.
\]
\end{corollary}

\begin{corollary}
There is a rational isomorphism
\[
BP^*(BA) \lra{} \Prod{H \subseteq A} BP^*(BH)/I.
\]
\end{corollary}
\begin{proof}
The argument is similar to the proof of \myref{symcor}.
\end{proof}

\subsection{Artin induction for $BP$}
It is not hard to use the methods of the previous subsections to prove an Artin induction theorem for $BP$. We do this by transferring the result for $\E$ (due to \cite{hkr}) to $BP$. Stronger results than this have been known for some time (Theorem A in \cite{hkr} applies to all complex oriented theories).  

Let $\mathcal{A}(G)$ be the full subcategory of the orbit category of $G$ consisting of quotients of the form $G/A$, where $A$ is an abelian $p$-group.

\begin{proposition} \mylabel{artin}
For $G$ a good group, the canonical map
\[
BP^*(BG) \lra{} \lim_{A \in \mathcal{A(G)}}BP^*(BA)
\]
is a rational isomorphism.
\end{proposition}
\begin{proof}
Because $G$ is good there is no integral torsion in $BP^*(BG)$. This follows immediately from the injection
\[
BP^*(BG) \hookrightarrow \Prod{n}\E^*(BG).
\]
Character theory for good groups \cite{hkr} implies that the natural map
\[
\Prod{n}\E^*(BG) \lra{} \prod_{n>0}\lim_{A \in \mathcal{A(G)}}\E^*(BA)
\]
is injective and, since limits of split injections are split injections, the canonical map
\[
\lim_{A \in \mathcal{A(G)}}BP^*(BA) \lra{} \prod_{n>0}\lim_{A \in \mathcal{A(G)}}\E^*(BA)
\]
is injective. 

We are interested in the commutative diagram
\[
\xymatrix{ & BP^*(BG) \ar[r] \ar@<1ex>[d] & \lim_{A \in \mathcal{A(G)}}BP^*(BA) \ar[r] \ar@<1ex>[d] & C_{BP} \ar[r] \ar@<1ex>[d] & 0 \\
0 \ar[r] &\prod_{n>0}\E^*(BG) \ar[r]\ar@/^0.5pc/@{-->}[u] & \prod_{n>0}\lim_{A \in \mathcal{A(G)}}\E^*(BA) \ar[r]  \ar@/^0.5pc/@{-->}[u] & \Prod{n}C_{\E} \ar@/^0.5pc/@{..>}[u] \ar[r] & 0.}
\]

Since the left vertical arrow is injective and the second bottom horizontal arrow is injective, the natural map
\[
BP^*(BG) \lra{} \lim_{A \in \mathcal{A(G)}}BP^*(BA)
\]
is injective. This implies that the top sequence of maps is short exact. Thus it suffices to prove that $C_{BP}$ is torsion.

This implies that it suffices to show that $C_{\E}$ has exponent bounded independently of the height $n$. By Theorem A of \cite{hkr}, the abelian group $C_{\E}$ is torsion. 

Note the map of short exact sequences
\[
\xymatrix{\E^*(BG) \ar[r] \ar[d]^-{=} & \lim_{A \in \mathcal{A(G)}}\E^*(BA) \ar[r] \ar[d] & C_{\E} \ar[d] \\ \E^*(BG) \ar[r] & \Prod{A \subset G} \E^*(BA) \ar[r] & D_{\E},}
\]
where $D_{\E}$ is the cokernel. The first two vertical arrows are injections (the first is an equality), so the right vertical arrow is an injection. Thus it suffices to show that the exponent of the torsion in $D_{\E}$ is bounded independent of $n$.

To bound the torsion we apply character theory to height $1$. Since $C_{1}^*$ is a flat $\E^*$-module we get a short exact sequence
\[
C_{1}^* \otimes_{\E^*} \E^*(BG) \hookrightarrow \Prod{A \subset G} C_{1}^* \otimes_{\E^*}\E^*(BA) \lra{} C_{1}^* \otimes_{\E^*} D_{\E}.
\]
Now we may follow the argument in \myref{leveliso} to reduce to $p$-adic $K$-theory. This leads to considering the cokernel of the following injection
\[
\Prod{[\al]} K_{p}^*(BC(\im \al)) \hookrightarrow \Prod{A \subset G} \Prod{\beta} K_{p}^*(BA). 
\]
If we fix a conjugacy class $[\al \colon \Z_{p}^{n-1} \rightarrow G]$ then there is a map $K_{p}^*(BC(\im \al)) \lra{} K_{p}^*(BA)$ when the composite
\[
\Z_{p}^{n-1} \lra{\beta} A \subset G
\]
is conjugate to $\al$. In that case $A \subset C(\im \al)$. Let $i_H \colon A \subset G$ be the inclusion. The map out of the factor corresponding to $[\al]$ is
\[
K_{p}^*(BC(\im \al)) \lra{} \Prod{A \subset G} \Big ( \Prod{\{\beta | [i_H\beta] = [\al]\}} K_{p}^*(BA) \Big).
\] 
This map is an injection although it may not be the case that every abelian subgroup of $C(\im \al)$ is represented in the codomain. Abelian subgroups of $G$ that are conjugate to abelian subgroups of $C(\im \al)$ may appear as well. Also, subgroups may appear multiple times. However, \myref{diagonal} implies that repeat subgroups do not contribute to the exponent of the torsion in the cokernel. Finally, since there are only a finite number of abelian subgroups of $C(\im \al)$ (or abelian subgroups of $G$ that may be conjugated into $C(\im \al)$), the exponent of the torsion in the cokernel must be bounded by the maximal exponent of this finite number of options.
\end{proof}

This subsection is not just an exercise. Artin induction for $\E$ can be proved independently of Theorem A in \cite{hkr} by using character theory. The retract theorem now allows us to bootstrap to $BP$. It is not hard to imagine moving from there to any Landweber flat theory (or perhaps complex oriented). One might hope that this could provide an independent proof of Theorem A in \cite{hkr}, at least when $G$ is good and $X = *$. 


\appendix

\section{A proof of Proposition \ref{transfertorsion}, by Jeremy Hahn}

\acknowledgements{
I thank Danny Shi and Mike Hopkins for instilling me with enough practical knowledge of level structures to contribute to this project. \\}

We recall our notational conventions. Let $H$ be a finite abelian group, $n$ a positive integer, and $E_n$ a height $n$ Morava $E$-theory at the prime $p$.  We use $H^*$ to denote the Pontryagin dual group $\Hom(H,S^1)$.  In the ring $E_n^*(BH)$, $I$ will denote the ideal generated by the images of transfers from proper subgroups of $H$.  We let $R_n$ denote the quotient ring $E_n^*(BH)/I$. 

Algebro-geometrically, $E_n^*(BH)$ may be interpreted as global sections of the affine formal scheme $\Hom(H^*,\mathbb{G}_{E_n})$.  The closed subscheme with global sections $R_n=E_n^*(BH)/I$ is not as well-behaved as one might hope.  In particular, for almost no values of $H$ and $n$ is $R_n$ an integral domain.  Let $T_n$ denote the ideal of $p$-power torsion in $R_n$ (i.e., $x \in T_n$ if and only if $p^a x = 0$ for some positive integer $a$).  In \cite{ahs}, Ando, Hopkins, and Strickland identify the further quotient $R_n/T_n$ with an object of algebro-geometric significance: it is the global sections of the formal scheme of Drinfeld level structures $\text{Level}(H^*,\mathbb{G}_{E_n})$.  In \cite{drinfeld}, Drinfeld proved that $\Gamma \text{Level}(H^*,\mathbb{G}_{E_n})$ is a regular local ring, so $R_n/T_n$ is in fact very well-behaved.

The Ando-Hopkins-Strickland result allows one to identify $\mathbb{Q} \otimes E_n^*(BH)/I \text{ with } \mathbb{Q} \otimes \Gamma \text{Level}(H^*,\mathbb{G}_{E_n})$.  In this appendix, our goal will be to extend this to an isomorphism
$$\mathbb{Q} \otimes \prod_{n>0} E_n^*(BH)/I \cong \mathbb{Q} \otimes \prod_{n>0} \Gamma \text{Level}(H^*,\mathbb{G}_{E_n}).$$

The isomorphism follows immediately from the following proposition, whose proof occupies the remainder of the appendix. 

\begin{proposition} \label{appendixprop}
There is a positive integer $k$, depending on $H$ but \textbf{not} depending on $n$, such that $p^{k} T_n = 0 \subset R_n$. 
\end{proposition}

In other words, we claim the exponent of the $p$-power torsion in $E_n^*(BH)/I_n$ may be bounded independently of the height $n$.

We can make two immediate reductions.  First, note that $H$ may be written as a direct sum of a $p$-group and a group of order prime to $p$, the prime to $p$ part having no contribution to $E$-theory.  Furthermore, since $E_n^*(BH)$ is Noetherian, the exponent of the $p$-power torsion in $R_n$ is bounded for any fixed $n$.  In light of these facts, for the remainder of the appendix we assume that $H$ is a $p$-group and study $E_n^*(BH)$ for $n$ larger than the $p$-rank of $H$.

Since $E_n^*(BH) \cong \Gamma\Hom(H^*,\mathbb{G}_{E_n})$, there is a natural inclusion of $H^*$ into $E_n^*(BH)$.  This is the map that takes a character $\chi:H \rightarrow S^1$ to the first Chern class of the induced map $BH \rightarrow BS^1$.  Composing with the projection $E_n^*(BH) \rightarrow R_n$, we obtain a map $\phi:H^* \rightarrow R_n$.

\begin{lemma} \label{appendixdivis}
For each non-zero $\chi \in H^*$, $\phi(\chi)$ divides $p$ in $R_n$.
\end{lemma}

\begin{proof}
Suppose $\chi:H \rightarrow S^1$ has kernel $A \subset H$.  The quotient $H/A$ is isomorphic to a non-trivial subgroup of $S^1$. Assume that $H/A$ is cyclic of order $p^k$.  It follows that the kernel of $p^{k-1}\chi$ is an index $p$ subgroup $A'$ of $H$ containing $A$. The composite of the transfer and projection 
\[
E_n^*(BA') \rightarrow E_n^*(BH) \rightarrow R_n
\]
then sends $1$ to 
\[
\langle p^k \rangle (\phi(\chi)) = \frac{[p^{k}](\phi(\chi))}{[p^{k-1}](\phi(\chi))},
\]
by the corresponding formula for the transfer $e \subset \mathbb{Z}/p$ and naturality of transfer.
In particular, $\langle p^k \rangle (\phi(\chi))=0$ in $R_n$.
But $\langle p^k \rangle (\phi(\chi))$ is a power series in $\phi(\chi)$ with constant term $p$, so it follows that $\phi(\chi)$ divides $p$.
\end{proof}

Suppose now that
$$H \cong \mathbb{Z}/p^{m_1} \mathbb{Z} \times \mathbb{Z}/p^{m_2} \mathbb{Z} \times ... \times \mathbb{Z}/p^{m_{\ell}} \mathbb{Z},$$
where $m_1 \ge m_2 \ge ... \ge m_{\ell}$.  We proceed to prove Proposition \ref{appendixprop} by induction on $\ell$.

If $\ell=1$, then $H \cong \mathbb{Z}/p^{m_1} \mathbb{Z}$.  The case is covered in Section 4.1 of this paper, but see also \cite[Example 9.22]{ahs}.

For $\ell>1$, consider $$H' \cong \mathbb{Z}/p^{m_1} \mathbb{Z} \times \mathbb{Z}/p^{m_2} \mathbb{Z} \times ... \times \mathbb{Z}/p^{m_{\ell-1}} \mathbb{Z}$$ equipped with the natural inclusion $H' \subset H$.  Let $R'_n$ denote $E_n^*(BH')$ modulo its ideal of transfers, and let $T_n'$ denote the ideal of $p$-power torsion in $R'_n$.  Finally, we use the notation $\phi':(H')^* \rightarrow R'_n$ to denote the obvious analogue of $\phi$.

The power series ring $(R'_n)[[z]]$ is an $E^*_n$-algebra, so one may consider the element $[p^{m_{\ell}}](z)$, which we denote by $q(z)$.  We use $\bar{q}(z)$ to denote projection of $q(z)$ to $(R'_n/T'_n)[[z]]$.  For each $\chi \in (H')^*$ of order at most $p^{m_{\ell}}$, one has that $[p^{m_{\ell}}](\phi'(\chi))=0$.  Since $R'_n/T'_n$ is a regular local ring, and since $n$ is larger than the rank of $H$, it follows that $\bar{q}(z)$ is divisible by $$\prod_{\chi} (z-\phi'(\chi)),$$ where the product runs across all $\chi \in (H')^*$ of order at most $p^{m_{\ell}}$.  The quotient of $\bar{q}(z)$ by this product may, by the Weierstrass preparation theorem, be written as a unit times a monic polynomial $g(z)$.

The following key lemma was proved by Drinfeld \cite[Proof of Prop.~4.3]{drinfeld} and realized in topology by Strickland \cite[Section 26]{stricklandfp}, \cite[Section 7]{subgroups}:

\begin{lemma}[Drinfeld] \label{drinfeldpresentation}
$R_n/T_n \cong (R'_n/T'_n)[z]/g(z)$
\end{lemma}

The coefficients of $g(z)$ lie in $R_n'/T_n'$ and may be lifted arbitrarily to elements of $R_n'$.  This defines a non-canonical lift $\tilde{g}(z) \in R_n'[z]$ of $g(z)$.  Lemma \ref{drinfeldpresentation} then implies that $R_n/T_n$ is isomorphic to the quotient of $R_n'[z]$ by the ideal $(\tilde{g}(z),T_n')$ generated by $\tilde{g}(z)$ and the elements of $T_n'$.  Note that $R_n$ is already a quotient of $R_n'[z]$ by the K\"unneth isomorphism applied to $E_n^*(BH) \cong E_n^*(BH' \times B\mathbb{Z}/p^{m_{\ell}} \mathbb{Z})$ and elementary properties of transfers.

\begin{example} \label{appendixexample}
Suppose $H$ is the group $\mathbb{Z}/2 \times \mathbb{Z}/2 \times \mathbb{Z}/2$.  Then, for sufficiently large $n$, Lemma \ref{drinfeldpresentation} implies that 
$$R_n/T_n \cong E_n^*[[x,y,z]]/(\tilde{g}_1,\tilde{g}_2,\tilde{g}_3),$$
where 
\begin{itemize}
\item $\tilde{g}_1=\frac{[2](x)}{x} \text{ is an element defined in } E_n^*[[x]],$
\item $\tilde{g}_2=\frac{[2](y)}{y(y-x)} \text{ is a lift of an element uniquely defined in } E_n^*[[x,y]]/\tilde{g}_1, \text{ and}$
\item $\tilde{g}_3=\frac{[2](z)}{z(z-x)(z-y)(z-(x+_\mathbb{G} y)) }$ \text{ is a lift of an element defined in } $E_n^*[[x,y,z]]/(\tilde{g}_1,\tilde{g}_2).$
\end{itemize}
It follows that, if $w \in T_n$, then $w=a_1 \tilde{g}_1 + a_2 \tilde{g}_2 + a_3 \tilde{g}_3$ for some $a_1,a_2,a_3 \in R_n$.  The arguments below will show that, regardless of the height $n$, $16 \tilde{g}_1 = 16 \tilde{g}_2 = 16 \tilde{g}_3 = 0$, so $16w=0$.
\end{example}

\begin{remark} The reader familiar with \cite{ahs} should be wary of a slight typo in the statement of Proposition $9.15$.  The correct statement appears at the end of the proof of Proposition $9.15$ on page 29.  The difference between the stated and proved versions of Proposition $9.15$ is elucidated by the existence of elements such as $\tilde{g}_2$ in the above example, which annihilates $y-_{\mathbb{G}_{E_n}}x$ in $E^*(BH)/(\text{ann}(x),\text{ann(y)})$ but is not defined in $E_n^*(BH)$ itself.
\end{remark}

By the inductive assumption, there is an integer $k'$ such that $p^{k'} T'_n = 0 \subset R_n'$.  As a corollary of the comments preceding Example \ref{appendixexample}, we see that the complete proof of Proposition \ref{appendixprop} rests only on the following lemma.

\begin{lemma}
Let $h$ denote the image of $\tilde{g}(z)$ under the projection $\pi:R_n'[z] \rightarrow R_n$.  Then there is a positive integer $r$, independent of $n$, such that $p^r h = 0$.
\end{lemma}

\begin{proof}
Let $\chi:H' \rightarrow S^1$ denote any element of $(H')^*$.  We first claim that $\pi\left(z-\phi'(\chi)\right)$ divides $p$ in $R_n$.

\begin{itemize}

\item Let $\chi_1:H \rightarrow S^1$ denote the composite of the projection $H \rightarrow \mathbb{Z}/p^{m_{\ell}}$ with the inclusion $\mathbb{Z}/p^{m_{\ell}} \rightarrow S^1$.  Note that $\phi(\chi_1)=\pi(z)$.
\item Let $\chi_2:H \rightarrow S^1$ denote the composite of the projection $H \rightarrow H'$ with $\chi$.  Note that $\phi(\chi_2)=\pi(\phi'(\chi))$.
\end{itemize}

Now $(z-\phi'(\chi))$ is a unit times $z-_{\mathbb{G}_{E_n}} \phi'(\chi)$.  Furthermore,
$$\pi(z-_{\mathbb{G}_{E_n}} \phi'(\chi))=\phi(\chi_1)-_{\mathbb{G}_{E_n}} \phi(\chi_2)=\phi(\chi_1-\chi_2),$$
and by Lemma \ref{appendixdivis} the claim follows.

Turning to the proof of the lemma, recall that by definition $\tilde{g}(z) \prod_{\chi} (z-\phi'(\chi))$ has coefficients in $T_n'$, where the product ranges over all $\chi:H' \rightarrow S^1$ of order at most $p^{m_{\ell}}$.  It follows that $$\pi \left( p^{k'} \tilde{g}(z) \prod_{\chi} (z- \phi'(\chi)) \right)=0$$ in $R_n$.  Then by the claim 
$$\pi \left( p^{k'} \tilde{g}(z) \prod_{\chi} p \right)=0,$$
and so we may take $r$ to be the sum of $k'$ and the number of $\chi$ appearing in the product.
\end{proof}

\bibliographystyle{amsalpha}
\bibliography{mybib}

\end{document}